\documentclass[a4paper,12pt]{amsart}

\usepackage{enumerate,layout,amssymb,epsf,mathrsfs}
\usepackage{eepic,graphicx,amscd}
\usepackage{wrapfig,bbm,braket,enumitem}

\newtheorem{theorem}{Theorem}[section]

\newtheorem{lemma}[theorem]{Lemma}

\newtheorem{corollary}[theorem]{Corollary}

\theoremstyle{definition}

\newtheorem{definition}[theorem]{Definition}

\newtheorem{remark}[theorem]{Remark}

\newcommand{\Z}{{\mathbb Z}}
\newcommand{\N}{{\mathbb N}}

\newcommand{\R}{{\mathbb R}}
\newcommand{\B}{{\mathcal B}}

\renewcommand{\L}{{\mathcal L}}
\renewcommand{\P}{{\mathcal P}}

\DeclareMathOperator{\Orb}{Orb}
\DeclareMathOperator{\supp}{supp}
\DeclareMathOperator{\id}{id}
\DeclareMathOperator{\gr}{gr}

\numberwithin{equation}{section}

\begin{document}

\title[Relative, strictly ergodic model theorem]%
{A relative, strictly ergodic model theorem for infinite measure-preserving systems}

\author[H. Yuasa]{Hisatoshi Yuasa}
\address{Division of Science, Mathematics and Information, Osaka Kyoiku University,
4-698-1 Asahigaoka, Kashiwara, Osaka 582-8582, JAPAN.}

\email{hyuasa@cc.osaka-kyoiku.ac.jp}

\begin{abstract}
Every factor map between given ergodic, measure-preserving transformations on infinite Lebesgue spaces 
has a strictly ergodic, locally compact Cantor model. 
\end{abstract}



\maketitle

\section{Introduction}\label{intro}

A {\em Cantor set} is a totally disconnected, compact metric space without isolated points. As well 
known, such topological spaces are all homeomorphic to each other; see for example 
\cite[Theorem 2-97]{HockYoung}. A topological dynamical system on a Cantor set is called a 
{\em Cantor system}. A {\em Cantor model} of a measure-preserving system $\mathbf{Y}$ on a probability 
space is a Cantor system $\mathbf{X}$ which equipped with an invariant probability measure is 
measure-theoretically isomorphic to the system $\mathbf{Y}$. The Cantor system $\mathbf{X}$ is said to 
be {\em strictly ergodic} if it is {\em minimal}, i.e.\ every orbit is dense in $X$, and in addition, 
it is {\em uniquely ergodic}, i.e.\ it has a unique, invariant Borel probability measure. 
Jewett-Krieger Theorem~\cite{Jewett,Krieger} affirms that any ergodic, measure-preserving system on a 
Lebesgue probability space has a strictly ergodic Cantor model. 

B.~Weiss provided a relative, strictly 
ergodic model theorem \cite[Theorem~2]{W1} by restricting himself to stating some relevant lemmas 
without detailed proofs, which affirms that any factor map between arbitrary ergodic, 
measure-preserving systems on Lebesgue probability spaces has a strictly ergodic Cantor model. 
\begin{figure}
\begin{minipage}{0.32\linewidth}%
\centering
\includegraphics[scale=1.1]{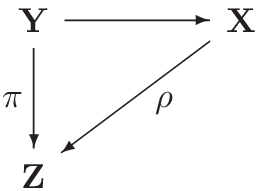}
\caption{}\label{triangle}
\end{minipage}\hfill
\begin{minipage}{0.32\linewidth}%
\centering
\includegraphics[scale=0.85]{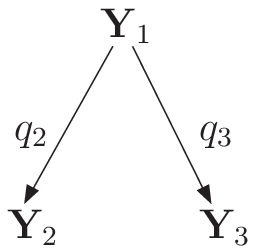}
\caption{}\label{inverted_tree}
\end{minipage}\hfill
\begin{minipage}{0.32\linewidth}%
\centering
\includegraphics[scale=0.85]{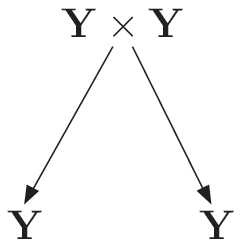}
\caption{}\label{product}
\end{minipage}
\end{figure}
F.~B\'eguin, S.~Crovisier and F.~Le~Roux~\cite[Section~1]{MR2875823} fully described a proof of the 
relative model theorem. One of our goals is to prove an infinite counterpart: 
\begin{theorem}\label{ICMT}
If $\pi$ is a measure-theoretical factor map from an ergodic, measure-preserving system 
$\mathbf{Y}$ on an infinite Lebesgue space to a strictly ergodic, locally compact Cantor system 
$\mathbf{Z}$, then there exist strictly ergodic, locally compact Cantor model $\mathbf{X}$ of the 
ergodic system $\mathbf{Y}$ and open and proper, topological factor map $\rho$ from $\mathbf{X}$ to 
$\mathbf{Z}$ for which the diagram in Figure~{\rm \ref{triangle}} commutes in the category of 
measure-preserving systems. 
\end{theorem}

A {\em locally compact Cantor set} \cite{D} is a totally disconnected, locally compact (non-compact) 
metric space $X$ without isolated points. The one-point compactification of such 
a topological space is exactly a Cantor set. If $S:X \to X$ is a homeomorphism, then 
$\mathbf{X}:=(X,S)$ is referred to as a {\em locally compact Cantor system}. The locally compact 
Cantor system $\mathbf{X}$ is said to be {\em strictly ergodic} if it is minimal, i.e.\ the orbit 
$\Orb_S(x) := \Set{S^kx|k \in \Z}$ of any point $x \in X$ is dense in $X$, and in addition, it is 
{\em uniquely ergodic}, i.e.\ an $S$-invariant Radon measure is unique up to scaling. This definition 
of strict ergodicity clearly extends that of unique ergodicity for Cantor systems. The system 
$\mathbf{X}$ is called a {\em locally compact Cantor model} of a measure-preserving system $\mathbf{Y}$ 
if the system $\mathbf{X}$ equipped with an invariant Radon measure is measure-theoretically isomorphic 
to the system $\mathbf{Y}$. F.~B\'eguin, S.~Crovisier and F.~Le~Roux \cite[Theorem~A.3]{MR2875823} also 
showed that any factor map between ergodic, measure-preserving systems on {\em probability} spaces has 
a ``strictly ergodic'', locally compact Cantor model, although they took into account unique 
ergodicity {\em only within the probability measures}. 

Combining Theorem~\ref{ICMT} with \cite[Theorem~4.6]{Y}, we obtain an immediate consequence that if 
$\mathbf{Y}$ and $\mathbf{Z}$ are ergodic, measure-preserving systems on infinite Lebesgue 
spaces and $\pi$ is a factor map from $\mathbf{Y}$ to $\mathbf{Z}$ then there exist strictly ergodic, 
locally compact Cantor models $\mathbf{X}$ and $\mathbf{W}$ of $\mathbf{Y}$ and $\mathbf{Z}$, 
respectively, and an open and proper, factor map $\rho$ from $\mathbf{X}$ to $\mathbf{W}$ so that a 
square diagram: 
\begin{center}
\includegraphics{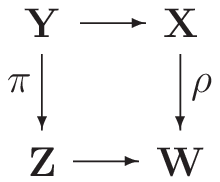}
\end{center}
commutes in the category of measure-preserving systems. We then say that $\pi:\mathbf{Y} \to 
\mathbf{Z}$ has a strictly ergodic, locally compact Cantor model $\rho:\mathbf{X} \to \mathbf{W}$. 

The above-mentioned consequence has a paraphrase in terms of {\em almost minimal} Cantor 
systems~\cite{D}. A Cantor 
system is said to be almost minimal if it has a unique fixed point and any other point has a dense 
orbit. An almost minimal Cantor system is regarded as the one-point compactification of a minimal, 
locally compact Cantor system. If the minimal, locally compact Cantor system is uniquely ergodic, then 
the associated almost minimal Cantor system is said to be {\em bi-ergodic}, as the point mass on a fixed 
point is always invariant for the almost minimal Cantor system. The bi-ergodicity is dealt 
with also in \cite[Theorem~A.3]{MR2875823}, although their notion is actually different from ours as 
mentioned above. The paraphrase is now stated as follows. {\it If $\mathbf{Y}$ and $\mathbf{Z}$ are 
ergodic, measure-preserving systems on infinite Lebesgue spaces then any factor map 
$\pi:\mathbf{Y} \to \mathbf{Z}$ has a bi-ergodic, almost minimal Cantor model 
$\hat{\rho}:\hat{\mathbf{X}} \to \hat{\mathbf{W}}$.} It is now clear how to describe a paraphrase of 
Theorem~\ref{ICMT} in terms of almost minimal Cantor systems. 

B.~Weiss~\cite{W1} also gave a sufficient condition for a diagram in the category of ergodic 
measure-preserving systems on probability spaces to have a strictly ergodic Cantor model. 
Another goal of ours is to prove its infinite counterpart: 
\begin{theorem}\label{categorical_consequence}
In the category of ergodic, measure-preserving systems on infinite Lebesgue spaces, any diagram 
including no portions with shape of Figure~{\rm \ref{inverted_tree}} has a strictly ergodic, locally 
compact Cantor model. Actually, no diagrams including portions with shape of 
Figure~{\rm \ref{product}} have strictly ergodic, locally compact Cantor models. 
\end{theorem}

We will prove Theorems~\ref{ICMT} and \ref{categorical_consequence} in Sections~\ref{the_proof} and 
\ref{pf_for_thm1.2}, respectively, and in particular, accomplish a proof of Theorem~\ref{ICMT} almost 
along the line of \cite[Sub-lemma~A.1]{MR2875823}. In virtue of \cite{Y,Y5}, almost all devices are 
actually ready to achieve our goals. We will review them in Sections~\ref{se}-\ref{towers}. 

Any relations among measurable sets, any properties of maps between measure spaces, etc.\ are always 
understood to hold up to sets of measure zero. 

\section{Strict ergodicity}\label{se}

Let $(Y,\B,\mu)$ be an {\em infinite Lebesgue space}, i.e.\ a measure space isomorphic to a measure 
space $\R$ endowed with Lebesgue measure together with the $\sigma$-algebra of Lebesgue measurable 
subsets. Set $\B_0=\Set{A \in \B|0 < \mu(A) < \infty}$. The measure space $(Y,\B,\mu)$ has a 
countable {\em base} $\mathcal{E} \subset \B_0$, i.e.\ 
\begin{itemize}
\item
$\mathcal{E}$ {\em generates} the $\sigma$-algebra $\B$, i.e.\ $\B$ is 
the completion of the smallest $\sigma$-algebra including $\mathcal{E}$ with respect to $\mu$;
\item
$\mathcal{E}$ {\em separates points on} $Y$, i.e.\ only one of any distinct two points in $Y$ is 
included in some set belonging to $\mathcal{E}$.
\end{itemize}
The author refers the readers to \cite{Aa,G} for fundamental properties of a Lebesgue 
space. 

A bi-measurable bijection $T:Y \to Y$ is said to be {\em measure-preserving} if 
$\mu(T^{-1}E)=\mu(E)$ for all sets $E \in \B$. 
The measure $\mu$ is then said to be {\em $T$-invariant}, or simply, {\em invariant}. We refer to 
$\mathbf{Y}:=(Y,\B,\mu,T)$ as an {\em infinite measure-preserving system}. The system 
$\mathbf{Y}$ is said to be {\em ergodic} if the measure of any invariant set is zero or full, 
or equivalently, any $T$-invariant, measurable function on $Y$ is constant. 

Another infinite measure-preserving system $\mathbf{Z}=(Z,\mathcal{C},\nu,U)$ is called a 
{\em factor} of the system $\mathbf{Y}$ if there exists a measurable surjection 
$\phi:Y \to Z$ for which $\mu \circ \phi^{-1} = \nu$ and $\phi \circ T = U \circ \phi$. 
The map $\phi$ is then called a {\em factor map} from $\mathbf{Y}$ to $\mathbf{Z}$. If in addition 
$\phi$ is injective, then $\phi$ is called a {\em isomorphism} and $\mathbf{Y}$ is said to be 
{\em isomorphic} to $\mathbf{Z}$. If the system $\mathbf{Y}$ is ergodic, then the factor $\mathbf{Z}$ 
is necessarily ergodic, because if a measurable function $f:Z \to \R$ is $U$-invariant 
then a measurable function $f \circ \phi$ is $T$-invariant. 

We say that $\mathbf{X}=(X,S)$ is a {\em topological dynamical system} if $S$ is a homeomorphism on 
a $\sigma$-compact and locally compact metric space $X$. Given another topological dynamical 
system $\mathbf{W}=(W,V)$, a continuous surjection $\rho:X \to W$ is called a {\em factor map} from 
$\mathbf{X}$ to $\mathbf{W}$ if $\rho \circ S = V \circ \rho$. If, in addition, the continuous 
surjection $\rho$ is injective, then we say that $\rho$ is an {\em isomorphism} and $\mathbf{X}$ is 
{\em isomorphic} to $\mathbf{W}$. A positive Borel measure $\lambda$ on the $\sigma$-compact and 
locally compact metric space $X$ is called a {\em Radon measure} if $\lambda(K) < \infty$ for all 
compact subsets $K$ of $X$, so that any Radon measure is $\sigma$-finite. 
If the homeomorphism $S$ has a unique, up to scaling, invariant Radon measure, then the 
system $\mathbf{X}$ is said to be {\em uniquely ergodic}. If, in addition, the homeomorphism $S$ is 
minimal, then $\mathbf{X}$ is said to be {\em strictly ergodic}. 
If the metric space $X$ is a Polish space, then an ergodic decomposition  (see for example 
\cite[2.2.9]{Aa}) shows that the unique ergodicity implies the ergodicity. Actually, if $\lambda$ is 
the unique $S$-invariant Radon measure, then the ergodic decomposition guarantees that there exist a 
probability space $(\Omega,m)$ and $\sigma$-finite, ergodic Radon measures 
$\Set{\lambda_\omega|\omega \in \Omega}$ such that for any Borel set $A \subset X$,
\begin{itemize}
\item
a function $\Omega \to \R,\omega \mapsto \lambda_\omega(A)$ is measurable;
\item
$\lambda(A)=\int_\Omega \lambda_\omega(A) d m(\omega)$. 
\end{itemize}
Then, for any $\omega \in \Omega$, there exists a constant $c_\omega$ such that 
$\lambda_\omega = c_\omega \lambda$. Hence, the measure $\lambda$ is ergodic. 

Assume that $\mathbf{X}=(X,S)$ is a locally compact Cantor system; see Section~\ref{intro}. 
Let $\hat{X}$ denote the one-point compactification $X \cup \Set{\omega_X}$ of the locally compact 
Cantor set $X$. Define a homeomorphism $\hat{S}:\hat{X} \to \hat{X}$ by 
\[
\hat{S}x=
\begin{cases}
Sx & \textrm{ if } x \in X; \\
\omega_X & \textrm{ if } x = \omega_X.
\end{cases}
\]
We shall refer to $\hat{\mathbf{X}}:=(\hat{X},\hat{S})$ as the {\em one-point compactification} of 
the locally compact Cantor system $\mathbf{X}$. If the system $\mathbf{X}$ is minimal, then the system 
$\hat{\mathbf{X}}$ is almost minimal; see Section~\ref{intro}. Clearly, the one-point compactification 
provides us with a one-to-one correspondence between the class of minimal, locally compact Cantor 
systems and that of almost minimal Cantor systems. 

If $\phi:X \to W$ is a continuous surjection between locally compact 
Cantor sets, then we let $\hat{\phi}$ denote a surjection from $\hat{X}$ to $\hat{W}$ which is 
defined in the same way as $\hat{S}$ defined above. If $\hat{\phi}$ is continuous, then $\phi$ must 
be {\em proper}. A continuous map $\phi:X \to W$ is said to be proper if a map 
$\phi \times \id_Z: X \times Z \to W \times Z, (x,z) \mapsto (\phi(x),z)$ is closed for any 
topological space $Z$; see \cite[Definition~1 in \S 10.1]{bourbaki}. This condition is equivalent to 
each of the following conditions:
\begin{itemize}
\item
the inverse image $\phi^{-1}(C)$ of any compact subset $C$ of the locally compact Cantor set $W$ is 
compact in the locally compact Cantor set $X$;
\item
the map $\phi$ is closed and the inverse image $\phi^{-1} \Set{w}$ of any point $w \in W$ is compact.
\end{itemize}
See \cite[Appendix~C]{akin2017chain}. This equivalence allows us to verify that if a continuous 
surjection $\phi: X \to W$ is proper then the surjection $\hat{\phi}:\hat{X} \to \hat{W}$ is 
continuous. If the continuous surjection $\phi$ is a factor map between minimal, locally compact 
Cantor systems, then it is natural to assume that $\phi$ is proper as well. See also \cite{M}. 
\begin{lemma}[{\cite[a portion of Proposition~3.3]{Y}}]\label{char_unq}
Assume that a locally compact Cantor system $\mathbf{X}=(X,S)$ has a compact open set 
$K \subset X$ whose orbit $\bigcup_{n \in \Z} S^n(K)$ coincides with $X$. Then, 
\begin{itemize}
\item
$\mathbf{X}$ has an invariant Radon measure$;$
\item\label{char_unq_erg}
$\mathbf{X}$ is uniquely ergodic if and only if the following two conditions hold$:$
\begin{itemize}
\item
$\# (\Orb_S(x) \cap K) = \infty$ for all $x \in X;$
\item
to any compact open set $A \subset X$ and $\epsilon > 0$ correspond $c \ge 0$ and $m \in \N$ 
such that for all $x \in K$,
\[
S_n \mathbbm{1}_K(x) \ge m \Rightarrow 
\left| \frac{S_n \mathbbm{1}_A(x)}{S_n \mathbbm{1}_K(x)} - c \right| < \epsilon,
\]
where $\mathbbm{1}_A$ is the indicator function of $A$ and for a function $f$ on $X$, 
\[
S_n f = \sum_{i=-n}^{n-1}f \circ S^i.
\]
\end{itemize}
\end{itemize}
\end{lemma}

\begin{remark}
Let $\mathbf{X}$ be as in the hypothesis of Lemma~\ref{char_unq}. If $\mathbf{X}$ is uniquely ergodic 
under a unique, up to scaling, invariant Radon measure $\lambda$, then its minimality is equivalent to 
saying that the support $\supp(\lambda)$ of $\lambda$ coincides with $X$, which is 
the smallest closed subset of $X$ whose complement has measure zero with respect to $\lambda$. 
Compare this fact with \cite[Theorem~6.17]{Walters1}. 

There exist bi-ergodic, for whose definition see Section~\ref{intro}, almost minimal subshifts over 
finite alphabets~\cite{Y4}, which arise from non-primitive substitutions. 
\end{remark}

\begin{corollary}\label{factor_of_unq_LCCM}
If $\phi$ is a proper factor map from a strictly ergodic, locally compact Cantor system 
$\mathbf{X}=(X,S)$ to a locally compact Cantor system $\mathbf{W}=(W,V)$, then the system 
$\mathbf{W}$ is strictly ergodic. 
\end{corollary}

\begin{proof}
Any point $w \in W$ equals $\phi(x)$ for some point $x \in X$. Since $\phi$ is continuous, 
\[
W=\phi(X)=\phi \left(\overline{\Orb_S(x)}\right) \subset \overline{\phi(\Orb_S(x))} = 
\overline{\Orb_V(w)},
\]
and hence $\mathbf{W}$ is minimal. 

Fix a nonempty compact open subset $L$ of $W$. Since $\phi$ is proper and continuous, an inverse 
image $\phi^{-1}(L)$ is compact and open. Since $\mathbf{W}$ is minimal, it holds that 
$\bigcup_{i \in \Z}V^i L=W$ and hence $\bigcup_{i \in \Z}S^i \phi^{-1}(L)=X$. Let $\epsilon > 0$ and 
compact open subset $B$ of $W$ be arbitrary. Since $\mathbf{X}$ is uniquely ergodic, in virtue of 
Lemma~\ref{char_unq}, there exist $c \ge 0$ and $m \in \N$ such that for all $x \in \phi^{-1}(L)$,
\[
S_n \mathbbm{1}_{\phi^{-1}(L)}(x) \ge m \Rightarrow \left| \frac{S_n \mathbbm{1}_{\phi^{-1}(B)}(x)}
{S_n\mathbbm{1}_{\phi^{-1}(L)}(x)} -c \right| < \epsilon,
\]
which implies that for all $y \in L$,
\[
V_n \mathbbm{1}_L(y) \ge m \Rightarrow \left| \frac{V_n \mathbbm{1}_B(y)}{V_n\mathbbm{1}_L(y)} -c 
\right| < \epsilon,
\]
because $S_n \mathbbm{1}_{\phi^{-1}(B)}=(V_n\mathbbm{1}_B) \circ \phi$. Finally, Lemma~\ref{char_unq} 
again guarantees the unique ergodicity of the system $\mathbf{W}$. 
\end{proof}

\section{Symbolic factors associated with partitions}\label{part_symbl}

Let $\mathbf{Y}=(Y,\B,\mu,T)$ be an ergodic, infinite measure-preserving system. 
Simply by a {\em partition} of $Y$, we mean an ordered, finite family 
$\alpha=\Set{A_1,A_2,\dots,A_m}$ of nonempty, measurable subsets of $Y$ satisfying that 
\begin{itemize}
\item
$m \ge 2$; 
\item
$\alpha$ is a partition of $Y$ in the usual sense;
\item
$\mu(A_i)=\infty$ if and only if $i = 1$.
\end{itemize}
Each member of $\alpha$ is called an {\em atom}. In particular, a single atom $A_1$ 
and the other atoms $A_2,\dots,A_m$ are called {\em infinite atom} and {\em 
finite atoms}, respectively. Set $K_\alpha=Y \setminus A_1$, which is called a {\em finite support} 
of $\alpha$. We let $\mathfrak{A}_\alpha$ denote the set of {\em subscripts} of atoms of the partition 
$\alpha$, i.e.\ $\mathfrak{A}_\alpha=\Set{1,2,\dots,m}$. 

Let $\beta=\Set{B_1,B_2,\dots,B_n}$ be another partition of $Y$. If each atom of the partition $\beta$ 
is included in an atom of the partition $\alpha$, then the partition $\beta$ is called a 
{\em refinement} of the partition $\alpha$. It is then designated by $\beta \succ \alpha$. If, in 
addition, the partitions $\alpha$ and $\beta$ have their finite supports in common, i.e.\ 
$K_\beta=K_\alpha$, then the refinement is designated by $\beta \succcurlyeq \alpha$. If $m = n$, then 
the partitions $\alpha$ and $\beta$ have a distance defined by 
\[
d(\alpha,\beta) = \sum_{i \ne 1}\mu(A_i \triangle B_i),
\]
which is a complete metric in the sense that if $(\alpha_i)_{i \in \N}$ is a Cauchy sequence of 
partitions in the distance $d$ then there exists a partition $\alpha$ to which the sequence 
$(\alpha_i)_{i \in \N}$ converges in the distance $d$. 

A {\em join} $\alpha \vee \beta$ of the partitions $\alpha$ and $\beta$ is defined to be a partition: 
\[
\Set{A \cap B|A \in \alpha, B \in \beta}.
\]
Clearly, the notion of join can be extended to any finite number of partitions. Given partitions 
$\alpha_1,\alpha_2,\dots,\alpha_k$ of $Y$, it is natural to regard 
$\mathfrak{A}_{\alpha_1 \vee \alpha_2 \vee \dots \vee \alpha_k}$ as a subset of a product set 
$\prod_{i=1}^k \mathfrak{A}_{\alpha_i}$. For example, an infinite atom of 
$\bigvee_{i=1}^k \alpha_i$ has subscript $(1,1, \dots, 1)$. We let $\mathcal{F}_\alpha$ denote 
the algebra generated by $\bigcup_{k=1}^\infty \alpha_{-k}^{k-1}$, where 
\begin{equation*}
\alpha_\ell^k = T^{-\ell}\alpha \vee T^{-(\ell+1)}\alpha \vee \dots \vee T^{-k}\alpha
\end{equation*}
for integers $k,\ell$ with $\ell \le k$. Also, set 
\begin{equation}\label{alpha^T}
\alpha^T=\Set{\bigcup_i T^{k_i}A_{\ell_i}| k_i \in \Z, A_{\ell_i} \in \alpha \cap \B_0, \textrm{ and } 
\ell_i \ne \ell_j \textrm{ if } i \ne j}.
\end{equation}
Remark that $\# (\alpha \cap \B_0) < \infty$. Given a set $E \in \B$, we use notation 
$E \overset{\epsilon}{\in} \alpha^T$ to mean that $\mu(E \triangle F) \le \epsilon$ for some set 
$F \in \alpha^T$. 
\begin{remark}\label{conti_1-block_to_k-block}
Assume that $\# \alpha = \# \beta$ in other words $\mathfrak{A}_\alpha = \mathfrak{A}_\beta$. Since 
\[
\mu \left(\bigcap_{j=-k+1}^{k-1}T^{-j}A_{i_j} \triangle \bigcap_{j=-k+1}^{k-1}T^{-j}B_{i_j} 
\right) \le \sum_{j=-k+1}^{k-1} \mu(A_{i_j} \triangle B_{i_j}) \le (2k-1)d(\alpha,\beta)
\]
for arbitrary $k \in \N$ and $\Set{i_j \in \Set{1,2,\dots,n}|-k+1 \le j \le k-1}$, 
it follows that if $\mathfrak{A}_{\alpha_{-k+1}^{k-1}} = \mathfrak{A}_{\beta_{-k+1}^{k-1}}$, viewed as 
subsets of $\prod_{i=-k+1}^{k-1} \mathfrak{A}_\alpha$ and $\prod_{i=-k+1}^{k-1} \mathfrak{A}_\beta$ 
respectively, then 
$d(\alpha_{-k+1}^{k-1},\beta_{-k+1}^{k-1}) \le (2k-1) \cdot \# \alpha_{-k+1}^{k-1} \cdot 
d(\alpha,\beta)$. 
This implies that $(\alpha_i)_{-k+1}^{k-1} \to \alpha_{-k+1}^{k-1}$ in $d$ as $i \to \infty$ if 
$\mathfrak{A}_{(\alpha_i)_{-k+1}^{k-1}}= \mathfrak{A}_{\alpha_{-k+1}^{k-1}}$ for every $i \in \N$ and 
if $\alpha_i \to \alpha$ in $d$ as $i \to \infty$. 
\end{remark}

Again, let $\alpha=\Set{A_1,A_2,\dots,A_m}$ be a partition of $Y$. 
The set $\mathfrak{A}_\alpha$ is now regarded as a (finite) alphabet, and so, its elements are 
sometimes called letters. Let $\Lambda$ and $(\mathfrak{A}_\alpha)^\ast$ denote the empty word and 
the set of (possibly empty) words over the alphabet $\mathfrak{A}_\alpha$, respectively. An infinite 
product space $(\mathfrak{A}_\alpha)^\Z$ is a Cantor set under the product topology of discrete 
topologies. The family of {\em cylinder subsets} is a base for the topology. A cylinder subset 
$[u.v]$ of $(\mathfrak{A}_\alpha)^\Z$ is associated with words $u,v \in (\mathfrak{A}_\alpha)^\ast$ 
in such a way that 
\[
[u.v] = 
\Set{x=(x_i)_{i \in \Z} \in (\mathfrak{A}_\alpha)^\Z|x_{[-|u|,|v|)} 
:= x_{-|u|}x_{-|u|+1} \dots x_{|v|-1}=uv}, 
\]
where $|u|$ is the length of the word $u$. A cylinder subset $[\Lambda.v]$ is abbreviated to $[v]$. 
Consider a homeomorphism 
$S_\alpha:(\mathfrak{A}_\alpha)^\Z \to (\mathfrak{A}_\alpha)^\Z,(x_i)_{i \in \Z} 
\mapsto (x_{i+1})_{i \in \Z}$, which is the so-called left shift on $(\mathfrak{A}_\alpha)^\Z$. 
Define a map $\phi_\alpha:Y \to (\mathfrak{A}_\alpha)^\Z$ by 
$T^iy \in A_{\phi_\alpha(y)_i}$ for every $i \in \Z$. The map $\phi_\alpha$ is measurable. 
Since for any $i \in \Z$ and $y \in Y$, a point $T^{i+1}y$ belongs to both of atoms 
$A_{\phi_\alpha(Ty)_i}$ and $A_{\phi_\alpha(y)_{i+1}}$, we obtain that 
$\phi_\alpha(Ty)=S_\alpha \phi_\alpha(y)$ for all $y \in Y$. Set 
$\hat{\lambda_\alpha}=\mu \circ {\phi_\alpha}^{-1}$, which is an $S_\alpha$-invariant, infinite Borel 
measure on $(\mathfrak{A}_\alpha)^\Z$. Set $\hat{X_\alpha}=\supp(\hat{\lambda_\alpha})$. Since 
$\hat{\lambda_\alpha}$ is $S_\alpha$-invariant, so is $\hat{X_\alpha}$. Set 
\begin{align}
\L(\alpha) &= \Set{u \in (\mathfrak{A}_\alpha)^\ast| \bigcap_{i=1}^{|u|}T^{-(i-1)}A_{u_i} \ne 
\emptyset}; \label{itinerary_words} \\
\L(\alpha)^\prime &= \Set{u \in \L(\alpha) | u \ne 1^{|u|} :=
\underbrace{11 \dots 1}_{|u| \textrm{ letters}}}; \notag \\
\L_n(\alpha)^\prime &= \Set{u \in \L(\alpha)^\prime | \vert u \vert =n}, \notag
\end{align}
where we have to recall that ``a measurable set is nonempty'' should be interpreted as ``a 
measurable set has a positive measure''. 
Since it follows from definition of $\supp(\hat{\lambda_\alpha})$ that given a sequence 
$x=(x_i)_{i \in \Z} \in (\mathfrak{A}_\alpha)^\Z$, 
\begin{equation}\label{meas_fact_subshift}
x \in \hat{X_\alpha} \Leftrightarrow x_{[-n,n)} \in \L(\alpha) \textrm{ for all } n \in \N, 
\end{equation}
we see that a sequence $1^\infty$, which consists of the letter $1$ only, belongs to $X_\alpha$. 
Since $T$ is ergodic, we obtain that 
$\hat{\lambda_\alpha}(\Set{1^\infty})=\mu\left(\bigcap_{i \in \Z} T^{-i}A_1\right)=0$. In addition, 
it is easy to see that $0 < \hat{\lambda_\alpha}([u.v]_{X_\alpha})<\infty$ whenever 
$uv \in \L(\alpha)^\prime$, where $[u.v]_{X_\alpha}=[u.v] \cap X_\alpha$. Hence, 
$\hat{\lambda_\alpha}$ is $\sigma$-finite. Let $\hat{\mathbf{X}_\alpha}$ denote a subshift 
$(\hat{X_\alpha},\hat{S_\alpha})$, where $\hat{S_\alpha}$ is the restriction of the left shift 
$S_\alpha$ to $\hat{X_\alpha}$. The map $\phi_\alpha$ is a factor map from the system $\mathbf{Y}$ to 
an ergodic, infinite measure-preserving system $(\hat{X_\alpha},\hat{\lambda_\alpha},\hat{S_\alpha})$. 
To see that $\hat{\lambda_\alpha}$ is non-atomic, let $x \in \hat{X_\alpha}$ be arbitrary. We may 
assume that $x \in \hat{X_\alpha} \setminus \Set{1^\infty}$, so that 
$\hat{\lambda_\alpha}(\Set{x}) < \infty$. If $x$ is shift-periodic, i.e.\ $\hat{S_\alpha}^px=x$ for 
some $p \in \N$, then $\bigcup_{i=0}^{p-1}T^i{\phi_\alpha}^{-1}\Set{x}$ is a $T$-invariant set, 
because $T^p {\phi_\alpha}^{-1}\Set{x}={\phi_\alpha}^{-1}\Set{x}$, and hence, 
$0 \le \hat{\lambda_\alpha}(\Set{x})=\mu({\phi_\alpha}^{-1}\Set{x}) \le 
\mu\left(\bigcup_{i=0}^{p-1}T^i{\phi_\alpha}^{-1}\Set{x}\right)=0$ since $T$ is ergodic. If $x$ is 
aperiodic, then $W:={\phi_\alpha}^{-1}\Set{x}$ is a {\em wandering set} for $T$, i.e.\ 
$T^iW \cap T^jW = \emptyset$ if $i \ne j$. Since $T$ is conservative \cite[Proposition~1.2.1]{Aa}, 
we have $\hat{\lambda_\alpha}(\Set{x})=\mu(W)=0$. 

Suppose that partitions $\alpha=\Set{A_1,\dots,A_m}$ and $\beta=\Set{B_1,\dots,B_n}$ have a relation 
$\beta \succ \alpha$. Define a map: 
\begin{equation}\label{map_f}
f_{\alpha,\beta}:\mathfrak{A}_\beta \to \mathfrak{A}_\alpha
\end{equation}
by $B_j \subset A_{f_{\alpha,\beta}(j)}$ for each $j \in \mathfrak{A}_\beta$. 
Define a map $\phi_{\alpha,\beta}:\hat{X_\beta} \to \hat{X_\alpha}, (x_i)_{i \in \Z} \mapsto 
(f_{\alpha,\beta}(x_i))_{i \in \Z}$, which is a factor map from the subshift $\hat{\mathbf{X}_\beta}$ 
to the subshift $\hat{\mathbf{X}_\alpha}$. It is readily verified that 
$\phi_{\alpha,\beta}(1^\infty)=1^\infty$ and that 
$\hat{\lambda_\beta} \circ {\phi_{\alpha,\beta}}^{-1}=\hat{\lambda_\alpha}$ because 
$\phi_\alpha = \phi_{\alpha,\beta} \circ \phi_\beta$. It is clear from definition that 
$\phi_{\alpha,\beta}$ maps each cylinder subset of $X_\beta$ onto a cylinder subset of 
$\hat{X_\alpha}$, so that $\phi_{\alpha,\beta}$ is an open map. 

\begin{definition}[{\cite[Definition~4.1]{Y}}]
A set $A \in \B_0$ is said to be {\em uniform} relative to a set 
$K \in \B_0$ if to any number $\epsilon > 0$ corresponds $m \in \N$ such that for any point $y \in K$,
\begin{equation*}\label{def_uniform_rel}
T_n\mathbbm{1}_K(y) \ge m \Rightarrow \left| 
\frac{T_n\mathbbm{1}_A(y)}{T_n\mathbbm{1}_K(y)} - \frac{\mu(A)}{\mu(K)} \right| < \epsilon.
\end{equation*}

A partition $\alpha$ of $Y$ is said to be {\em uniform} if every set belonging to a family 
$\B_0 \cap \bigcup_{n=1}^\infty \alpha_{-n}^{n-1}$ is uniform relative to the finite support 
$K_\alpha$ of the partition $\alpha$. 

If a uniform partition $\alpha$ of $Y$ is such that for all 
$x \in \hat{X_\alpha} \setminus \{1^\infty\}$,
\begin{equation}\label{inf_often_hitting}
\# (\Orb_{\hat{S_\alpha}}(x) \cap K_{\phi_\alpha(\alpha)})= \infty,
\end{equation}
where $K_{\phi_\alpha(\alpha)}=\bigcup_{i \ne 1} \phi_\alpha(A_i)$, then the partition $\alpha$ is 
said to be {\em strictly uniform}. 
\end{definition}

It is easy to see that a partition $\alpha$ is uniform if and only if there exists a sequence 
$1 \le n_1<n_2< \ldots$ of integers for which every set belonging to a family 
$\B_0 \cap \bigcup_{k=1}^\infty \alpha_{-n_k+1}^{n_k-1}$ is uniform relative to $K_\alpha$. 

\begin{lemma}[{\cite[Lemma~4.2]{Y}}]\label{str_unf_iff_str_erg}
Let $\alpha$ be a partition of $Y$. Then, the partition $\alpha$ is strictly uniform if and only if 
the subshift $\hat{\mathbf{X}_\alpha}$ is almost minimal and bi-ergodic. 
\end{lemma}

\section{Towers}\label{towers}

Let $\mathbf{Y}=(Y,\B,\mu,T)$ be again an ergodic, infinite measure-preserving system. 
A partition $\mathfrak{t}=\Set{T^jB_i|1 \le i \le k,0 \le j < h_i}$ 
of $Y$ is called a {\em tower} with {\em base} $B(\mathfrak{t}) := 
\bigcup_{i=1}^k B_i$. An atom of the tower $\mathfrak{t}$ is called a {\em level}. 
Since the tower $\mathfrak{t}$ is required to be a partition of $Y$, there necessarily exists 
a unique $i_0$ for which $\mu(B_{i_0})=\infty$. Then, it is necessary that $h_{i_0}=1$. 
We may always assume that $i_0=1$. The unique level $B_1$ of infinite 
measure is referred to as an {\em infinite level} of $\mathfrak{t}$. Another tower 
$\mathfrak{t}^\prime$ is called a {\em refinement} of the tower $\mathfrak{t}$ if 
$B(\mathfrak{t}^\prime) \subset B(\mathfrak{t})$ and $\mathfrak{t}^\prime \succ \mathfrak{t}$. 

Each subfamily $\mathfrak{c}_i:=\Set{T^jB_i|0 \le j < h_i}$ 
is called a {\em column} of the tower $\mathfrak{t}$. We refer to $B_i$ and $h_i$ 
as the {\em base} and {\em height} of the column $\mathfrak{c}_i$, respectively. 
With abuse of terminology, we also say that $\bigcup_{j=0}^{h_i-1}T^jB_i$ is a 
column of $\mathfrak{t}$. The column $\mathfrak{c}_i$ is said to be {\em principal} if 
$i \ne 1$. Let $h(\mathfrak{t})$ denote the least value of heights of principal columns of 
$\mathfrak{t}$.

A set $\Set{T^jy|0 \le j < h_i}$ with $y \in B_i$ and $1 \le i \le k$ is called a {\em fiber} 
of the tower $\mathfrak{t}$. Any set of the form $\Set{T^iy|m \le i \le n}$ with $m \le n$ is called a 
{\em section} of the orbit $\Orb_T(y)$, so that any fiber is a section. The word 
$\phi_\alpha(y)_{[m,n]}$ over the alphabet $\mathfrak{A}_\alpha$ is called the $\alpha$-{\em name} 
of the section. Sections $\Set{T^iy|m \le i \le n}$ and 
$\Set{T^iy^\prime|m^\prime \le i \le n^\prime}$ are said to be {\em consecutive} if 
$T(T^ny)=T^{m^\prime}y^\prime$. 

Let $K \in \B_0$. We say that a tower $\mathfrak{t}$ is 
{\em $K$-standard}~\cite[Definition~2.1]{Y} if 
\begin{itemize}
\item
$K$ has a nonempty intersection with each principal column of $\mathfrak{t}$;
\item
$K$ is a union of levels included in the principal columns of $\mathfrak{t}$.
\end{itemize}
These conditions imply that the number of those points in a fiber which belong to the set $K$ is 
positive and constant on each column of the tower $\mathfrak{t}$ and that the number is zero for all 
the fibers on the infinite level of the tower $\mathfrak{t}$. Let $h_K(\mathfrak{t})$ and 
$H_K(\mathfrak{t})$ denote the least and greatest values of such numbers over fibers on principal 
columns of $\mathfrak{t}$, respectively. 

\begin{lemma}[{\cite[Proposition~2.5]{Y}}]\label{suff_tall}
If $C$ is a measurable subset of $K$, $0 < \epsilon < \mu(K)$ and $M \in \N$, then there exists 
$N \in \N$ such that if a $K$-standard tower $\mathfrak{t}$ satisfies $h(\mathfrak{t}) > N$ then the 
union of those fibers $\Set{T^jy|0 \le j < h_y}$ of the tower $\mathfrak{t}$, for which 
\[
\left| \frac{\sum_{i=0}^{h_y-1} \mathbbm{1}_C(T^iy)}
{\sum_{i=0}^{h_y-1} \mathbbm{1}_K(T^iy)} - 
\frac{\mu(C)}{\mu(K)}\right| < \epsilon \textrm{ and } \sum_{i=0}^{h_y-1}\mathbbm{1}_K(T^iy) \ge M,
\]
covers at least $\mu(K)-\epsilon$ of $K$ in measure. 
\end{lemma}

Apart from measure-preserving systems, we then discuss towers for an almost minimal Cantor system 
$\hat{\mathbf{Z}}=(\hat{Z},\hat{U})$, which is the one-point compactification of a minimal, locally 
compact Cantor system $\mathbf{Z}=(Z,U)$. Put $\hat{Z}=Z \cup \Set{\omega_Z}$ as done in 
Section~\ref{se}, which is a Cantor set. Take a clopen neighborhood $C$ of $\omega_Z$ so that $C \ne 
\hat{Z}$. Since 
$\omega_Z$ is an accumulation point of the forward orbit $\Set{\hat{U}^kz|k \in \N}$ of any point 
$z \in \hat{Z}$~\cite[Theorem~1.1]{HPS}, the {\em return time function}: 
\[
r_C:C \to \N, z \mapsto \min\Set{k \in \N|{\hat{U}}^kz \in C}
\]
is well-defined. For each $k \in \N$, an inverse image: 
\[
{r_C}^{-1}\Set{k}=\left(C \cap {\hat{U}}^{-k}C\right) \setminus \bigcup_{i=1}^{k-1}{\hat{U}}^{-i}C
\]
is open as the right hand side looks. Since the clopen set $C$ is compact, the function $r_C$ is 
bounded. Then, put $r_C(C)=\Set{h_1 < h_2 < \dots < h_c}$. Since $C$ includes a fixed point $\omega_Z$ 
of $\hat{U}$, it is necessary that $h_1=1$. Since $\hat{Z} \setminus C \ne \emptyset$, the almost 
minimality of $\hat{U}$ implies that $C \setminus {\hat{U}}^{-1}(C) \ne \emptyset$, so that we 
must have $c \ge 2$. It is readily verified that 
\begin{equation}\label{K-R}
\P=\Set{{\hat{U}}^jC_i| 0 \le j < h_i, 1 \le i \le c},
\end{equation}
where $C_i={r_C}^{-1}\Set{h_i}$, is a family of mutually disjoint, nonempty, clopen subsets of 
$\hat{Z}$. It is a partition of $\hat{Z}$ in the usual sense, because a clopen subset 
$\bigcup_{A \in \P}A$ of $\hat{Z}$ is $\hat{U}$-invariant. A partition of the form \eqref{K-R} is 
called a {\em Kakutani-Rohlin partition}~\cite{HPS} (abbreviated to K-R partition) of the almost 
minimal Cantor system $\hat{\mathbf{Z}}$. The above discussion guarantees that every almost minimal 
Cantor system possesses a K-R partition. 

We do not require in general that $h_1,\dots,h_c$ are mutually distinct, but always require that 
$h_1=1$ (and hence $\omega_Z \in C_1$). A K-R partition is topological counterpart of a tower for a 
measure-preserving system, as the clopen set $C_1$ corresponds to the infinite level of a tower. It is 
possible to naturally import terminology and notation defined in the category of measure-theoretical 
systems into that of almost minimal Cantor systems. For example, each subfamily 
$\Set{{\hat{U}}^jC_i|0 \le j < h_i}$ is called a column, and the clopen set $C_1$ is called an 
infinite level. 

We then discuss symbolic factors of the almost minimal Cantor system $\hat{\mathbf{Z}}$. Suppose 
that $\alpha=\Set{A_1,A_2,\dots,A_m}$ with $m \ge 2$ is a partition of $\hat{Z}$ into nonempty, 
clopen subsets such that $\omega_Z \in A_1$. Actually, this is a definition of a {\em partition} of 
$\hat{Z}$. The atom $A_1$ corresponds to the infinite atom of a partition of a measure space, which 
is discussed in Section~\ref{part_symbl}. 
So, we let $K_\alpha$ denote the {\em finite support} $\bigcup_{i = 2}^m A_i$ of 
the partition $\alpha$. Define $\L(\alpha)$ by literally interpreting 
\eqref{itinerary_words} after replacing $T$ with $\hat{U}$. With abuse of notation, define a subshift: 
\begin{equation}\label{top_fact_subshift}
\hat{X_\alpha}=\Set{x=(x_i)_{i \in \Z} \in (\mathfrak{A}_\alpha)^\Z|x_{[-n,n)} \in \L(\alpha) 
\textrm{ for all } n \in \N},
\end{equation}
where $\mathfrak{A}_\alpha=\Set{1,2,\dots,m}$ again. Define a factor map 
$\phi_\alpha:\hat{Z} \to \hat{X_\alpha}$ by ${\hat{U}}^iz \in A_{\phi_\alpha(z)_i}$ 
for all $i \in \Z$. We can also define a factor map 
$\phi_{\alpha,\beta}:\hat{X_\beta} \to \hat{X_\alpha}$ as in Section~\ref{part_symbl} if $\beta 
\succ \alpha$. The subshift $\hat{\mathbf{X}_\alpha}:=(\hat{X_\alpha},\hat{S_\alpha})$ is almost 
minimal; recall the proof of Corollary~\ref{factor_of_unq_LCCM}. 

We then discuss K-R partitions of the subshift $\hat{\mathbf{X}_\alpha}$ which are associated with 
{\em return words}; see \cite{DHS,Du} for the minimal subshifts and \cite{Y4,Y5} for the almost 
minimal subshifts. They induce K-R partitions of the almost minimal Cantor system $\hat{\mathbf{Z}}$. 
Choose $u,v \in \L(\alpha)$ so that $uv \in \L(\alpha) \setminus \Set{\Lambda}$. A word 
$w \in \L(\alpha)$ is called a return word to $u.v$ if $uwv \in \L(\alpha)$ and in addition $uv$ 
occurs in $uwv$ exactly twice as prefix and suffix. Let $\mathcal{R}_n(\alpha)$ denote the set of 
return words to $1^n.1^n$. It is clear that $1 \in \mathcal{R}_n(\alpha)$. In virtue of 
\cite[Section~3]{Y4} and \cite[Lemma~5.2]{Y5}, we know that 
\begin{itemize}
\item
$\mathcal{R}_n(\alpha)$ is a finite set for all $n \in \N$;
\item
for every $n \in \N$, each word $w \in \mathcal{R}_{n+1}(\alpha)$ has a unique decomposition into 
words belonging to $\mathcal{R}_n(\alpha)$;
\item
if $w \in \mathcal{R}_n(\alpha) \setminus \Set{1}$, then $|w| > 2n$, $w_{[1,n]}=1^n, w_{n+1} \ne 1$ 
and $w_{[|w|-(n-1),|w|]}=1^n, w_{|w|-n} \ne 1$. 
\item
\begin{equation}\label{num_of_not_1_to_inf}
\lim_{n \to \infty}\min \Set{ \vert w \vert_{\neg 1}| w \in \mathcal{R}_n(\alpha) \setminus \Set{1}} 
= +\infty,
\end{equation}
where  $|w|_{\neg 1} = \# \Set{1 \le i \le \vert w \vert | w_i \in \mathfrak{A}_\alpha \setminus 
\Set{1}}$. 
\end{itemize}
To see \eqref{num_of_not_1_to_inf}, assume that it is not the case. Find a sequence 
$x \in \hat{X_\alpha}$ of the form $1^\infty.w1^\infty$ with a word $w$ whose first and last letter is 
neither $1$, where the dot means a separation between the nonnegative and negative coordinates. Since 
the only accumulation point of $\Orb_{\hat{S_\alpha}}(x)$ is $1^\infty$,  we obtain that 
$\hat{X_\alpha}=\Orb_{\hat{S_\alpha}}(x) \cup \Set{1^\infty}$, 
because $\hat{\mathbf{X}_\alpha}$ is almost minimal. This leads to a contradiction that $\mathbf{Z}$ 
has a nonempty open wandering set. 

Now, it is readily verified that for each $n \in \N$, 
\[
\P_n(\alpha) := \Set{{\hat{S_\alpha}}^j([1^n.w1^n]_{\hat{X_\alpha}})|w \in \mathcal{R}_n(\alpha), 
0 \le j < |w|}
\]
is a K-R partition of $\hat{\mathbf{X}_\alpha}$ with base 
$B(\P_n(\alpha))=[1^n.1^n]_{\hat{X_\alpha}}$. Set 
\begin{align}
\tilde{\mathfrak{t}}_n(\alpha) &={\phi_\alpha}^{-1}\P_n(\alpha) 
:=\Set{{\phi_\alpha}^{-1}(A)|A \in \P_n(\alpha)} \notag \\
&= \Set{{\hat{U}}^j {\phi_\alpha}^{-1}([1^n.w1^n]_{\hat{X_\alpha}}) | 
w \in \mathcal{R}_n(\alpha), 0 \le j < |w|}, \label{actual_shape_of_tnalpha}
\end{align}
which is a K-R partition of $\hat{\mathbf{Z}}$. Observe that 
\begin{equation}\label{inv_img_KR}
{\phi_\alpha}^{-1}([1^n.w1^n]_{\hat{X_\alpha}}) =
\bigcap_{i=-n}^{\vert w \vert +n-1}{\hat{U}}^{-i}A_{(1^nw1^n)_{i+n+1}} \textrm{ and } 
B(\tilde{\mathfrak{t}}_n(\alpha)) = \bigcap_{j=-n}^{n-1} {\hat{U}}^{-j}A_1.
\end{equation}
For every $n \in \N$, 
$\tilde{\mathfrak{t}}_{n+1}(\alpha)$ is a refinement of $\tilde{\mathfrak{t}}_n(\alpha)$, 
because 
\begin{itemize}
\item
$B(\tilde{\mathfrak{t}}_{n+1}(\alpha))=
{\phi_\alpha}^{-1}([1^{n+1}.1^{n+1}]_{\hat{X_\alpha}}) \subset 
{\phi_\alpha}^{-1}([1^n.1^n]_{\hat{X_\alpha}})= B(\tilde{\mathfrak{t}}_n(\alpha))$; 
\item
letting $w=u_1u_2 \dots u_k \ (u_i \in \mathcal{R}_n(\alpha))$ be a unique decomposition of a given 
word $w \in \mathcal{R}_{n+1}(\alpha)$, we have that if $|u_1 \dots u_{i-1}| \le j < |u_1 \dots u_i|$ 
for $1 \le i \le k$ then ${\hat{S_\alpha}}^j([1^{n+1}.w1^{n+1}]_{\hat{X_\alpha}}) \subset 
{\hat{S_\alpha}}^{j- |u_1 \dots u_{i-1}|}([1^n.u_i1^n]_{\hat{X_\alpha}})$ and hence 
${\hat{U}}^j {\phi_\alpha}^{-1}([1^{n+1}.w1^{n+1}]_{\hat{X_\alpha}}) \subset {\hat{U}}^{j- |u_1 
\dots u_{i-1}|} {\phi_\alpha}^{-1}([1^n.u_i1^n]_{\hat{X_\alpha}})$, 
where we use a convention that $|u_1 \dots u_0|=0$.
\end{itemize}

\begin{lemma}\label{name_const_tall_tower}
The refining sequence $\Set{\tilde{\mathfrak{t}}_n(\alpha)}_{n \in \N}$ of $K_\alpha$-standard 
K-R partitions of the almost minimal Cantor system $\hat{\mathbf{Z}}$ satisfies that 
\begin{enumerate}
\item\label{proper_heights_to_inf}
$h_{K_\alpha}(\tilde{\mathfrak{t}}_n(\alpha)) \to + \infty$ as $n \to \infty;$
\item\label{constant_names}
for each $n \in \N$, the $\alpha$-names of fibers of $\tilde{\mathfrak{t}}_n(\alpha)$ are 
constant on each column.
\end{enumerate}

In addition, if a partition $\beta$ of $\hat{Z}$ implements a relation that $\beta \succ \alpha$, 
then for each $n \in \N$, $\tilde{\mathfrak{t}}_n(\beta)$ is a refinement of 
$\tilde{\mathfrak{t}}_n(\alpha)$. 
\end{lemma}

\begin{proof}
Property~\eqref{proper_heights_to_inf} follows from \eqref{num_of_not_1_to_inf}. 
Property~\eqref{constant_names} follows from \eqref{actual_shape_of_tnalpha} and the first half of 
\eqref{inv_img_KR}. Actually, each word $w \in \mathcal{R}_n(\alpha)$ is the $\alpha$-name of fibers 
on the column with base ${\phi_\alpha}^{-1}([1^n.w1^n]_{\hat{X_\alpha}})$. 

Then, let $\beta$ be as in the hypothesis. Since $f_{\alpha,\beta}(1)=1$, we have that 
$[1^n.1^n]_{\hat{X_\beta}} \subset {\phi_{\alpha,\beta}}^{-1}([1^n.1^n]_{\hat{X_\alpha}})$, so that 
$B(\tilde{\mathfrak{t}}_n(\beta)) ={\phi_\beta}^{-1}([1^n.1^n]_{\hat{X_\beta}}) \subset 
{\phi_\alpha}^{-1}( [1^n.1^n]_{\hat{X_\alpha}}) = B(\tilde{\mathfrak{t}}_n(\alpha))$, since 
$\phi_\alpha=\phi_{\alpha,\beta} \circ \phi_\beta$. What remains to be shown is that each level of 
$\P_n(\beta)$ is included in some level of ${\phi_{\alpha,\beta}}^{-1}\P_n(\alpha)$, since 
$\phi_\alpha=\phi_{\alpha,\beta} \circ \phi_\beta$. Since $f_{\alpha,\beta}(1)=1$, for any word 
$w \in \mathcal{R}_n(\beta)$, there exist unique words $u_1,u_2,\dots, u_k \in \mathcal{R}_n(\alpha)$ 
so that 
\begin{itemize}
\item
$|w|=|u_1u_2 \dots u_k|$;
\item
$f_{\alpha,\beta}(w_i)=(u_1u_2 \dots u_k)_i$ for every integer $i$ with $1 \le i \le |w|$. 
\end{itemize}
It follows from the conditions that 
\[
[1^n.w1^n]_{\hat{X_\beta}} \subset {\phi_{\alpha,\beta}}^{-1}
\left([1^n.u_1u_2 \dots u_k1^n]_{\hat{X_\alpha}}\right),
\]
and hence, 
\[
{S_\beta}^j\left([1^n.w1^n]_{\hat{X_\beta}}\right) \subset {S_\beta}^{j- |u_1 \dots u_{i-1}|}
{\phi_{\alpha,\beta}}^{-1}\left([1^n.u_i1^n]_{\hat{X_\alpha}}\right)
\]
for any integer $j$ with $0 \le j < |w|$ if $|u_1 \dots u_{i-1}| \le j < |u_1 \dots u_i|$ for 
$1 \le i \le k$. This completes the proof. 
\end{proof}

\begin{remark}
In the proof of Lemma~\ref{name_const_tall_tower}, it is necessary that $k=1$ if $\beta \succcurlyeq 
\alpha$. 
\end{remark}

Next, take a countable base $\Set{D_1,D_2,\dots}$ for the topology of the locally compact Cantor set 
$Z$ which consists of compact open subsets of $Z$. Form a sequence $\Set{\alpha_i}_{i \in \N}$ of 
partitions of $\hat{Z}$ by setting 
$\alpha_i=\Set{\hat{Z} \setminus D_1,D_1} \vee \Set{\hat{Z} \setminus D_2,D_2} \vee \dots \vee 
\Set{\hat{Z} \setminus D_i,D_i}$. Clearly, for each $i \in \N$, the partition $\alpha_{i+1}$ refines 
the partition $\alpha_i$, i.e.\ $\alpha_{i+1} \succ \alpha_i$. 
Set $\tilde{\mathfrak{t}}_i={\phi_{\alpha_i}}^{-1}\P_1(\alpha_i)$. 
In view of Lemma~\ref{name_const_tall_tower}, we have a refining sequence 
$\Set{\tilde{\mathfrak{t}}_i}_{i \in \N}$ of K-R partitions of the almost minimal Cantor system 
$\hat{\mathbf{Z}}$. The construction shows that 
\begin{itemize}
\item
for all $i \in \N$ and integer $j$ with $j \ge i$, the compact open subset $D_i$ of $Z$ is a union of 
levels in the principal columns of $\tilde{\mathfrak{t}}_j$;
\item
each principal column of $\tilde{\mathfrak{t}}_i$ has a nonempty intersection with $D_1$, because 
it is clear for $\tilde{\mathfrak{t}}_1$ and $\tilde{\mathfrak{t}}_i$ is a refinement of 
$\tilde{\mathfrak{t}}_1$ for each $i \in \N$.
\end{itemize}
Set $\gamma_1 = \tilde{\mathfrak{t}}_1$. Let $K_{\gamma_1}$ be as above, i.e.\ 
the union of principal columns of $\tilde{\mathfrak{t}}_1$. For each integer $i$ with $i > 1$, set 
$\gamma_i = \Set{\hat{Z} \setminus K_{\gamma_1},C| C \in \tilde{\mathfrak{t}}_i, C \subset 
K_{\gamma_1}}$. 
\begin{lemma}\label{seq_gen_top}
The sequence $\Set{\gamma_i}_{i \in \N}$ of partitions of $\hat{Z}$ into clopen subsets satisfies that
\begin{itemize}
\item
$\gamma_1 \preccurlyeq \gamma_2 \preccurlyeq \dots;$
\item
$\bigcup_{i,j \in \N} (\gamma_i)_{-j}^{j-1}$ is a base for the topology of $\hat{Z}$.
\end{itemize}
\end{lemma}

\begin{proof}
The first property is clear. The second property follows from a fact that 
each principal column of $\tilde{\mathfrak{t}}_i$ has a nonempty intersection with $K_{\gamma_1}$. 
\end{proof}

\section{Proof of Theorem~\ref{ICMT}}\label{the_proof}

Let $\mathbf{Y}=(Y,\B,\mu,T)$ be as in Theorem~\ref{ICMT}.
\begin{definition}
Let $\alpha$ and $\beta=\Set{B_1,B_2,\dots,B_n}$ be partitions of $Y$ with $\# \alpha = \# \beta=n$. 
The {\em $\alpha$ $(2k-1)$-block empirical distribution over a 
section} with $\alpha$-name $w$ is said to be {\em within $\epsilon$ of the 
$\beta_{-k+1}^{k-1}$ distribution} \cite[Definition~4.3]{Y} if for every word 
$v \in \L_{2k-1}(\alpha)^\prime$, 
\[
\left\vert \frac{\# \Set{1 \le j < \vert w \vert |w_{(j-k,j+k)}=v}}%
{\# \Set{1 \le j < \vert w \vert |w_j \ne 1}} - 
\frac{\mu\left(\bigcap_{i=-k+1}^{k-1}T^{-i}B_{v_{i+k}}\right)}
{\mu(K_\beta)} \right\vert < \epsilon,
\]
where $w_{(j-k,j+k)} = \begin{cases} w_{[1,j+k)} & \textrm{ if } j-k < 0; \\
w_{(j-k,|w|]} & \textrm{ if } j+k > n+1 \end{cases}$. 

We say that the partition $\alpha$ has {\em $(H,k,\epsilon)$-uniformity} if the 
$\alpha$ $(2k-1)$-block distribution over every section having at least $H$ points in $K_\alpha$ is 
within $\epsilon$ of the $\alpha_{-k+1}^{k-1}$ distribution. 
\end{definition}

Let $\mathbf{Z}=(Z,U)$ and $\pi$ be as in Theorem~\ref{ICMT}, so that there exists  a unique, up to 
scaling, infinite, $U$-invariant Radon measure $\nu$ satisfying that $\mu \circ \pi^{-1} = \nu$. 
Let $\mathcal{C}$ be the completion of the Borel $\sigma$-algebra with respect to $\nu$. Let 
$\hat{\mathbf{Z}}=(\hat{Z},\hat{U})$ be as posterior to Lemma~\ref{suff_tall}. Extend the measure 
$\nu$ naturally to a unique, non-atomic measure $\hat{\nu}$ onto the Borel $\sigma$-algebra of 
$\hat{Z}$. Complete the Borel $\sigma$-algebra with respect to $\hat{\nu}$, which is denoted by 
$\hat{\mathcal{C}}$. The map $\pi: Y \to Z$ is naturally regarded as a factor map from $\mathbf{Y}$ to 
$(\hat{Z},\hat{\mathcal{C}},\hat{\nu},\hat{U})$. 

Let $\Set{\gamma_i}_{i \in \N}$ be a refining sequence of partitions of $\hat{Z}$ as in 
Lemma~\ref{seq_gen_top}. In virtue of Corollary~\ref{factor_of_unq_LCCM}, for each $i \in \N$, the 
subshift $\hat{\mathbf{X}_{\gamma_i}}$ is a bi-ergodic, almost minimal Cantor system. 
For each $i \in \N$, set $\beta_i=\pi^{-1}\gamma_i$. 
It follows that $\Set{\beta_i}_{i \in \N}$ is a sequence of partitions of $Y$ satisfying that 
$\beta_1 \preccurlyeq \beta_2 \preccurlyeq \dots$\ Let $K$ denote their common finite support, which 
equals $\pi^{-1}\left(K_{\gamma_i}\right)$ for all $i \in \N$. Since the minimality of $U$ implies 
that any 
nonempty open subset of $Z$ has positive measure with respect to $\nu$, it follows from 
\eqref{meas_fact_subshift} and \eqref{top_fact_subshift} that $\hat{X_{\beta_i}}=\hat{X_{\gamma_i}}$ 
for each $i \in \N$. This together with Lemma~\ref{str_unf_iff_str_erg} implies that each $\beta_i$ 
is strictly uniform. Also, observe that $\phi_{\gamma_i} \circ \pi = \phi_{\beta_i}$ for each 
$i \in \N$. Choose a countable base $\mathcal{E} =\Set{E_i}_{i \in \N} \subset \B_0$ for the measure 
space $(Y,\B,\mu)$ so that each member appears in $\mathcal{E}$ infinitely often. Setting 
$K_0 = K \cap E_1$ and 
\[
K_i = (K \cap T^{-i}E_1) \setminus \bigcup_{j=1}^{i-1}T^{-j}E_1 \textrm{ for } i \in \N,
\]
we have a disjoint union $\bigcup_{i=0}^\infty T^iK_i=E_1$, because $Y=\bigcup_{j=1}^\infty T^{-j}E_1$ 
by the ergodicity. This allows us to find a partition $\tau_1$ with finite support $K$ implementing an 
approximation: 
\begin{equation}\label{approx_of_E1_by_tau1}
E_1 \overset{1/2^3}{\in} (\tau_1)^T
\end{equation}
by means of a finite union of mutually disjoint translates of mutually disjoint atoms; recall 
\eqref{alpha^T}. Fix $n_1 \in \N$ and $\delta_1 > 0$ so that 
\begin{align}
\frac{1}{2^{n_1}} & < \frac{1}{3}\min 
\Set{\frac{\mu(A)}{\mu(K)}|A \in (\tau_1 \vee \beta_1) \cap \B_0}; \label{Step1.1} \\
\delta_1 & < \min \Set{ \frac{1}{4 \cdot 2^{n_1}}, \ \frac{\mu(K)}{4 \cdot 2^{n_1}}, \ \frac{1}{2^3}}. 
\label{Step1.2} 
\end{align}
We regard $\mathfrak{A}_{\tau_1 \vee \beta_1}$ as a subset of a product set 
$\mathfrak{A}_{\tau_1} \times \mathfrak{A}_{\beta_1}$, so that in particular an infinite atom of 
$\tau_1 \vee \beta_1$ has a subscript $(1,1)$. 
\begin{remark}
Since every tower which we will encounter in this proof is $K$-standard, we will suppress the term 
`$K$-standard' for the sake of simplicity. Also, every partition will have the set $K$ as its finite 
support. 
\end{remark}
We appeal to a series of lemmas. Recall that given a partition $\alpha$ of $Y$, 
$\mathcal{F}_\alpha$ denotes the algebra generated by $\bigcup_{k=1}^\infty \alpha_{-k}^{k-1}$. 
\begin{lemma}[Step~1]\label{lem_Step_1}
There exist partition $\alpha_{1,1}$, $\mathcal{F}_{\beta_1}$-measurable 
tower $\mathfrak{t}_1$ and $N_1 \in \N$ such that 
\begin{itemize}
\item
$\alpha_{1,1}  \succcurlyeq \beta_1;$
\item
$E_1 \overset{1/2^2}{\in} (\alpha_{1,1})^T;$
\item
if a partition $\alpha$ of $Y$ is such that the $\alpha$-name of any fiber of $\mathfrak{t}_1$ coincides 
with the $\alpha_{1,1}$-name of some fiber of $\mathfrak{t}_1$, then the partition $\alpha$ has the 
$(N_1,1,1/2^{n_1})$-uniformity. 
\end{itemize}
\end{lemma}

\begin{proof}
Applying Lemma~\ref{suff_tall} to the finite atoms of a partition $\tau_1 \vee \beta_1$, we obtain 
$H_1 \in \N$ such that if $\mathfrak{t}$ is a tower with $h(\mathfrak{t}) \ge H_1$ then the {\em good} 
fibers of $\mathfrak{t}$, i.e.\ the $\tau_1 \vee \beta_1$ $1$-block empirical distributions over the 
fibers are within $\delta_1$ of the $\tau_1 \vee \beta_1$ distribution, covers at least 
$\mu(K)-\delta_1$ of $K$ in measure. Lemma~\ref{name_const_tall_tower} allows us to find an 
$\mathcal{F}_{\beta_1}$-measurable tower $\mathfrak{t}_1^\prime$ with 
$h(\mathfrak{t}_1^\prime) \ge H_1$ such that the $\beta_1$-names of fibers are constant on each column 
of $\mathfrak{t}_1^\prime$. The tower $\mathfrak{t}_1^\prime$ has the form 
$\pi^{-1}\tilde{\mathfrak{t}}_{n^\prime}(\gamma_1)$. Within each principal column of 
$\mathfrak{t}_1^\prime$, copy the $(\tau_1 \vee \beta_1)$-name of a good fiber into bad fibers, 
provided that good fibers exist in the column. This yields a new partition $\alpha_{1,1}$ with 
$\mathfrak{A}_{\alpha_{1,1}} \subset \mathfrak{A}_{\tau_1 \vee \beta_1}$. In virtue of the 
above-mentioned property of the tower $\mathfrak{t}_1^\prime$, the copying procedure 
never changes $\beta_1$-names although it may change $\tau_1$-names. Since it follows from 
\eqref{Step1.1} and \eqref{Step1.2} that 
\[
\delta_1< \frac{\mu(K)}{4 \cdot 2^{n_1}} < \frac{1}{12}\min \Set{ \mu(A)| A \in (\tau_1 \vee \beta_1) 
\cap \B_0},
\]
we obtain that $\mathfrak{A}_{\alpha_{1,1}}=\mathfrak{A}_{\tau_1 \vee \beta_1}$. Since a given point 
has the $\tau_1$-name $1$ if and only if it has the $\beta_1$-name $1$, we see that 
$K_{\alpha_{1,1}}=K_{\tau_1 \vee \beta_1}=K$. In order to see that 
$\alpha_{1,1} \succcurlyeq \beta_1$, again in virtue of the above-mentioned property of 
$\mathfrak{t}_1^\prime$, it is sufficient to observe that any enlargement or 
reduction of each finite atom of $\tau_1 \vee \beta_1$ caused by the copying procedure is carried out 
within a unique atom of $\beta_1$. Moreover, for the same reason, the refinement 
$\alpha_{1,1} \succcurlyeq \beta_1$ is implemented so that 
$f_{\beta_1,\alpha_{1,1}}= f_{\beta_1,\tau_1 \vee \beta_1}$; recall \eqref{map_f}. Since 
$d(\alpha_{1,1},\tau_1 \vee \beta_1) \le \delta_1 < 1/2^3$, 
it follows from \eqref{approx_of_E1_by_tau1} that 
\[
E_1 \overset{1/2^2}{\in} (\alpha_{1,1})^T.
\]

Let $R_1$ denote the intersection of the set $K$ with the union of those principal columns of 
$\mathfrak{t}_1^\prime$ which do not include any good fibers. As a consequence of the copying 
procedure, we know that if a fiber on a principal column of $\mathfrak{t}_1^\prime$ is disjoint 
from $R_1$ then the fiber has $\alpha_{1,1}$ $1$-block empirical distribution within $\delta_1$ of 
the $\tau_1 \vee \beta_1$ distribution. Let us verify that this leads to the existence of $M_1 \in \N$ 
such that for any point $y \in K$, 
\begin{equation}\label{M1}
A \in \alpha_{1,1} \cap \B_0 \textrm{ and } T_n\mathbbm{1}_{K \setminus R_1}(y) \ge M_1 \Rightarrow 
\left\vert \frac{T_n\mathbbm{1}_{A \setminus R_1}(y)}{T_n\mathbbm{1}_{K \setminus R_1}(y)} - 
\frac{\mu(B)}{\mu(K)} \right\vert < \delta_1,
\end{equation}
where $B$ is a finite atom of $\tau_1 \vee \beta_1$ which shares a subscript with the finite atom $A$ 
of $\alpha_{1,1}$. Let $\mathscr{F}$ denote the family of good fibers of $\mathfrak{t}_1^\prime$. 
For each integer $m$ with $m \ge 3$, put
\[
\epsilon_m=\frac{4}{m-2} \cdot 
\frac{\max_{F \in \mathscr{F}}\# (F \cap K)}{\min_{F \in \mathscr{F}}\# (F \cap K)}.
\]
Fix an integer $m_0$ with $m_0 \ge 3$ so that 
\[
\epsilon_{m_0} < \delta_1 - \max \Set{ 
\left\vert \frac{\# (F \cap A)}{\# (F \cap K)}-\frac{\mu(B)}{\mu(K)}\right\vert | 
\begin{split}
& A \in \alpha_{1,1} \cap \B_0, B \in (\tau_1 \vee \beta_1) \cap B_0, F \in \mathscr{F}, \\
& \textrm{$B$ shares a subscript with $A$.}
\end{split}}.
\]
Set
\[
M_1= m_0 \max_{F \in \mathfrak{F}}\#(F \cap K),
\]
which is independent of the choice of the finite atom $A$. 
If $T_n\mathbbm{1}_{K \setminus R_1}(y) \ge M_1$ then there are integer $m_1$ with $m_1 \ge m_0$, 
nonempty subsets $C_{i_1}, C_{i_{m_1}}$ of good fibers of $\mathfrak{t}_1^\prime$, whose intersections 
with $K$ are nonempty, and fibers $C_{i_2},\dots,C_{i_{m_1-1}} \in \mathscr{F}$ for which 
$T_n\mathbbm{1}_{K \setminus R_1}(y) =\sum_{j=1}^{m_1} \# (C_{i_j} \cap K)$. Assume that 
$T_n\mathbbm{1}_{K \setminus R_1}(y) \ge M_1$ and $A \in \alpha_{1,1} \cap \B_0$ is arbitrary. Since
\begin{align*}
\frac{T_n\mathbbm{1}_{A \setminus R_1}(y)}{T_n\mathbbm{1}_{K \setminus R_1}(y)} &= 
\frac{\# (C_{i_1} \cap A) + \sum_{j=2}^{m_1-1} \# (C_{i_j} \cap A)+ \# (C_{i_{m_1}} \cap A)}%
{\# (C_{i_1} \cap K) + \sum_{j=2}^{m_1-1} \# (C_{i_j} \cap K)+ \# (C_{i_{m_1}} \cap K)} \\
&=\frac{\dfrac{\# (C_{i_1} \cap A)}{\sum_{j=2}^{m_1-1} \# (C_{i_j} \cap K)} + 
\dfrac{\sum_{j=2}^{m_1-1} \# (C_{i_j} \cap A)}{\sum_{j=2}^{m_1-1} \# (C_{i_j} \cap K)}
+ \dfrac{\# (C_{i_{m_1}} \cap A)}{\sum_{j=2}^{m_1-1} \# (C_{i_j} \cap K)}}
{\dfrac{\# (C_{i_1} \cap K)}{\sum_{j=2}^{m_1-1} \# (C_{i_j} \cap K)}+1 + 
\dfrac{\# (C_{i_{m_1}} \cap K)}{\sum_{j=2}^{m_1-1} \# (C_{i_j} \cap K)}}
\end{align*}
and since $\dfrac{\sum_{j=2}^{m_1-1} \# (C_{i_j} \cap A)}{\sum_{j=2}^{m_1-1} \# (C_{i_j} \cap K)} 
\le 1$, we obtain that
\begin{align*}
& \left\vert \frac{T_n\mathbbm{1}_{A \setminus R_1}(y)}{T_n\mathbbm{1}_{K \setminus R_1}(y)} - 
\frac{\sum_{j=2}^{m_1-1} \# (C_{i_j} \cap A)}{\sum_{j=2}^{m_1-1} \# (C_{i_j} \cap K)} \right\vert \\ 
& \le \left\vert \dfrac{\# (C_{i_1} \cap A)}{\sum_{j=2}^{m_1-1} \# (C_{i_j} \cap K)} - 
\dfrac{\# (C_{i_1} \cap K)}{\sum_{j=2}^{m_1-1} \# (C_{i_j} \cap K)} \cdot 
\dfrac{\sum_{j=2}^{m_1-1} \# (C_{i_j} \cap A)}{\sum_{j=2}^{m_1-1} \# (C_{i_j} \cap K)} \right. \\
& \ \quad \left. + \dfrac{\# (C_{i_{m_1}} \cap A)}{\sum_{j=2}^{m_1-1} \# (C_{i_j} \cap K)} 
- \dfrac{\# (C_{i_{m_1}} \cap K)}{\sum_{j=2}^{m_1-1} \# (C_{i_j} \cap K)} \cdot 
\dfrac{\sum_{j=2}^{m_1-1} \# (C_{i_j} \cap A)}{\sum_{j=2}^{m_1-1} \# (C_{i_j} \cap K)} 
\right\vert \\
& \le \epsilon_{m_1}.
\end{align*}
In view of an inequality:
\[
\min_{1 \le i \le k} \frac{a_i}{b_i} \le \frac{\sum_{i=1}^k a_i}{\sum_{i=1}^k b_i} \le 
\max_{1 \le i \le k} \frac{a_i}{b_i} \ \ (a_i,b_i > 0, k \in \N), 
\]
this allows us to obtain that 
\begin{align*}
\left\vert \frac{T_n\mathbbm{1}_{A \setminus R_1}(y)}{T_n\mathbbm{1}_{K \setminus R_1}(y)} - 
\frac{\mu(B)}{\mu(K)} \right\vert & \le \max_{F \in \mathscr{F}}\left\vert \frac{\# (F \cap A)}
{\# (F \cap K)} - \frac{\mu(B)}{\mu(K)}\right \vert + \epsilon_{m_1} \\ 
& \le  \max_{\substack{F \in \mathscr{F}, \ A \in \alpha_{1,1} \cap \B_0, \\ B \in (\tau_1 \vee \beta_1) 
\cap B_0, \\ \textrm{$B$ shares a subscript with $A$.}}}
\left\vert \frac{\# (F \cap A)}{\# (F \cap K)} - \frac{\mu(B)}{\mu(K)}\right \vert + \epsilon_{m_0} \\
& < \delta_1,
\end{align*}
which shows \eqref{M1}. Observe that in order to obtain \eqref{M1} we are not required to consider 
which good fiber each fiber $C_{i_j}$ is. This is one of facts which guarantee the last property of 
the lemma. 

On the other hand, since $R_1$ is $\mathcal{F}_{\beta_1}$-measurable, and hence, uniform 
relative to $K$, and since 
\[
\mu(R_1) \le \delta_1 < \frac{\mu(K)}{4 \cdot 2^{n_1}},
\]
there exists $N_1^\prime \in \N$ with 
\[
N_1^\prime \ge \left(1-\frac{1}{4 \cdot 2^{n_1}}\right)^{-1}M_1
\]
such that for any point $y \in K$,
\[
T_n\mathbbm{1}_K(y) \ge N_1^\prime \Rightarrow \frac{T_n\mathbbm{1}_{R_1}(y)}{T_n\mathbbm{1}_K(y)} 
< \frac{1}{4 \cdot 2^{n_1}},
\]
which implies further that 
\[
T_n\mathbbm{1}_{K \setminus R_1}(y) > \left(1-\frac{1}{4 \cdot 2^{n_1}} \right) \cdot 
T_n\mathbbm{1}_K(y) \ge M_1.
\]
It follows from the above argument that for any point $y \in K$, if 
$T_n\mathbbm{1}_K(y) \ge N_1^\prime$ then for every finite atom $A$ of $\alpha_{1,1}$,
\begin{align*}
\left| \frac{T_n\mathbbm{1}_A(y)}{T_n\mathbbm{1}_K(y)} - \frac{\mu(A)}{\mu(K)}\right| 
& \le \left| \frac{T_n\mathbbm{1}_{A \setminus R_1}(y)}{T_n\mathbbm{1}_K(y)} - 
\frac{\mu(B)}{\mu(K)}\right| + \frac{T_n\mathbbm{1}_{R_1}(y)}{T_n\mathbbm{1}_K(y)} 
+ \left| \frac{\mu(B)}{\mu(K)} - \frac{\mu(A)}{\mu(K)}\right| \\
&  < \left| \frac{T_n\mathbbm{1}_{A \setminus R_1}(y)}{T_n\mathbbm{1}_{K \setminus R_1}(y)}
\left(1-\frac{T_n\mathbbm{1}_{R_1}(y)}{T_n\mathbbm{1}_K(y)}\right) - 
\frac{\mu(B)}{\mu(K)}\right| + \frac{1}{2 \cdot 2^{n_1}} \\
& \le \left| \frac{T_n\mathbbm{1}_{A \setminus R_1}(y)}{T_n\mathbbm{1}_{K \setminus R_1}(y)}
 - \frac{\mu(B)}{\mu(K)}\right| + \frac{T_n\mathbbm{1}_{R_1}(y)}{T_n\mathbbm{1}_K(y)} + 
\frac{1}{2 \cdot 2^{n_1}} \\
& < \delta_1 + \frac{3}{4 \cdot 2^{n_1}} \\ 
& < \frac{1}{2^{n_1}}, 
\end{align*}
where $B$ is a finite atom of $\tau_1 \vee \beta_1$ which shares a subscript with the atom $A$ of 
$\alpha_{1,1}$. By using Lemma~\ref{name_const_tall_tower}, take an $\mathcal{F}_{\beta_1}$-measurable 
refinement $\mathfrak{t}_1$ of $\mathfrak{t}_1^\prime$ with $h_K(\mathfrak{t}_1) \ge N_1^\prime$. The 
tower $\mathfrak{t}_1$ has the form $\pi^{-1}\tilde{\mathfrak{t}}_n(\gamma_1)$. 
By argument similar to deducing \eqref{M1}, we can find an integer $N_1$ with $N_1 \ge N_1^\prime$ 
for which the last property of the lemma is valid. 
\end{proof}

In the same manner as \eqref{approx_of_E1_by_tau1}, find a partition $\tau_2$ with finite support 
$K$ implementing an approximation:
\begin{equation}\label{approx_of_E2_by_tau2}
E_2 \overset{1/2^4}{\in} (\tau_2)^T.
\end{equation}
The alphabet $\mathfrak{A}_{\alpha_{1,1} \vee \tau_2 \vee \beta_2}$ is regarded as a subset of a 
product set 
$\mathfrak{A}_{\alpha_{1,1}} \times \mathfrak{A}_{\tau_2} \times \mathfrak{A}_{\beta_2}$. 
Put 
\[
r_2=\#(\alpha_{1,1} \vee \tau_2 \vee \beta_2).
\]
Fix integer $n_2$ with $n_2 > n_1$ and real number $\delta_2 > 0$ so that 
\begin{align*}
\frac{1}{2^{n_2}} & < \frac{1}{3}\min \Set{\frac{\mu(A)}{\mu(K)}| 
A \in (\alpha_{1,1} \vee \tau_2 \vee \beta_2)_{-1}^1 \cap \B_0}; \\
\delta_2 & < \min \Set{ \frac{1}{3 \cdot 4 \cdot 2^{n_2} {r_2}^7}, \ %
\frac{\mu(K)}{3 \cdot 4 \cdot 2^{n_2} {r_2}^7}, \ \frac{1}{2^4}}.
\end{align*}
It is an immediate consequence that $n_2 \ge 2$. 

\begin{lemma}[Step~2]\label{scndstplmm}
There exist partitions $\alpha_{2,2}$ and $\alpha_{2,1}$, $\mathcal{F}_{\beta_2}$-measurable refinement 
$\mathfrak{t}_2$ of $\mathfrak{t}_1$ and integer $N_2$ with $N_2 > N_1$ such that 
\begin{itemize}
\item
$\alpha_{2,2} \succcurlyeq \beta_2$, $\alpha_{2,1} \succcurlyeq \beta_1$ and 
$\alpha_{2,2} \succcurlyeq \alpha_{2,1};$
\item
$E_1 \overset{1/2^2+1/2^3}{\in} (\alpha_{2,1})^T$ and $E_2 \overset{1/2^3}{\in} (\alpha_{2,2})^T;$
\item
$\# \alpha_{2,2}=r_2;$
\item
$d(\alpha_{2,1},\alpha_{1,1}) \le \delta_2 < 2^{-n_2};$
\item
the $\alpha_{2,1}$-name of any fiber of $\mathfrak{t}_1$ coincides with the 
$\alpha_{1,1}$-name of some fiber of $\mathfrak{t}_1;$
\item
if a partition $\alpha$ is such that the $\alpha$-name of any fiber of $\mathfrak{t}_2$ coincides 
with the $\alpha_{2,2}$-name of some fiber of $\mathfrak{t}_2$ then $\alpha$ has the 
$(N_2,2,1/(2^{n_2}{r_2}^4))$-uniformity.
\end{itemize}
\end{lemma}

\begin{proof}
Since $\mathfrak{t}_1$ has the form $\pi^{-1}\tilde{t}_n(\gamma_1)$, Lemmas~\ref{suff_tall} and 
\ref{name_const_tall_tower} allow us to obtain an $\mathcal{F}_{\beta_2}$-measurable refinement 
$\mathfrak{t}_2^\prime$ of $\mathfrak{t}_1$ such that 
\begin{itemize}
\item
the $\beta_2$-names of fibers of $\mathfrak{t}_2^\prime$ are constant on each column;
\item
{\em good} fibers of $\mathfrak{t}_2^\prime$, i.e.\ the 
$\alpha_{1,1} \vee \tau_2 \vee \beta_2$ $3$-block empirical distributions on the fibers are within 
$\delta_2$ of the $(\alpha_{1,1} \vee \tau_2 \vee \beta_2)_{-1}^1$ distribution, covers at least 
$\mu(K)-\delta_2$ of $K$ in measure.
\end{itemize}
In fact, when applying Lemma~\ref{suff_tall}, we may have to consider translates $T(A)$ or 
$T^{-1}(A)$ instead of finite atoms $A \in (\alpha_{1,1} \vee \tau_2 \vee \beta_2)_{-1}^1$ themselves. 
The tower $\mathfrak{t}_2^\prime$ has the form $\pi^{-1}\tilde{t}_{n^\prime}(\gamma_2)$. 
Since $\mathfrak{t}_2^\prime$ is a refinement of $\mathfrak{t}_1$, the $\alpha_{1,1}$-name of any fiber 
of $\mathfrak{t}_2^\prime$ is a concatenation of $\alpha_{1,1}$-names of fibers of $\mathfrak{t}_1$. 

We obtain a new partition $\alpha_{2,2}$ by copying the 
$(\alpha_{1,1} \vee \tau_2 \vee \beta_2)$-name of a good fiber in each column of 
$\mathfrak{t}_2^\prime$ into bad fibers in the same column, provided that good fibers exist in the 
column. The $\beta_2$-part (or, 
$\beta_2$-coordinate) of any $(\alpha_{1,1} \vee \tau_2 \vee \beta_2)$-name never changes under 
the copying procedure although the other parts may change. In particular, the copying procedure 
preserves $(\alpha_{1,1} \vee \tau_2 \vee \beta_2)$-name $(1,1,1)$. Recall that 
$K_{\beta_2}=K_{\alpha_{1,1}}=K_{\tau_2}=K$. These facts guarantee that 
$K_{\alpha_{2,2}}=K$. Since 
\[
d\left(\alpha_{1,1} \vee \tau_2 \vee \beta_2,\alpha_{2,2}\right) \le \delta_2 < \frac{1}{2^4},
\]
it follows from \eqref{approx_of_E2_by_tau2} that $E_2 \overset{1/2^3}{\in} (\alpha_{2,2})^T$. 
Since $\delta_2 < \mu(A)$ for all finite atoms $A$ of $(\alpha_{1,1} \vee \tau_2 \vee \beta_2)_{-1}^1$, 
we see that 
$\mathfrak{A}_{(\alpha_{2,2})_{-1}^1}=\mathfrak{A}_{(\alpha_{1,1} \vee \tau_2 \vee \beta_2)_{-1}^1}$ 
as well as $\mathfrak{A}_{\alpha_{2,2}}=\mathfrak{A}_{\alpha_{1,1} \vee \tau_2 \vee \beta_2}$. 
Since, as mentioned above, the $\beta_2$-part of any $(\alpha_{1,1} \vee \tau_2 \vee \beta_2)$-name 
never changes under the copying procedure, we see that $\alpha_{2,2} \succcurlyeq \beta_2$ and moreover 
$f_{\beta_2,\alpha_{2,2}}=f_{\beta_2,\alpha_{1,1} \vee \tau_2 \vee \beta_2}$. 
In virtue of Remark~\ref{conti_1-block_to_k-block}, we obtain that 
\[
d((\alpha_{1,1} \vee \tau_2 \vee \beta_2)_{-1}^1,(\alpha_{2,2})_{-1}^1) < 3{r_2}^3\delta_2 < 
\frac{1}{4 \cdot 2^{n_2}{r_2}^4}.
\]

The intersection $R_2$ of the set $K$ with the union of principal columns of $\mathfrak{t}_2^\prime$ 
which include no good fibers is uniform relative to $K$, because it is 
$\mathcal{F}_{\beta_2}$-measurable. Applying the same argument as we developed above in order to find 
$N_1^\prime$ in the proof of Lemma~\ref{lem_Step_1}, we obtain an integer $N_2^\prime$ with 
$N_2^\prime > N_1$ such that for any point $y \in K$, 
\begin{equation*}\label{unf_alpha22}
A \in (\alpha_{2,2})_{-1}^1 \cap \B_0 \textrm{ and } T_n\mathbbm{1}_K(y) \ge N_2^\prime \Rightarrow 
\left| \frac{T_n\mathbbm{1}_A(y)}{T_n\mathbbm{1}_K(y)} - \frac{\mu(A)}{\mu(K)}\right| < 
\frac{1}{2^{n_2} {r_2}^4}. 
\end{equation*}
Lemma~\ref{name_const_tall_tower} allows us to have an $\mathcal{F}_{\beta_2}$-measurable 
refinement $\mathfrak{t}_2$ of $\mathfrak{t}_2^\prime$ with $h_K(\mathfrak{t}_2) \ge N_2^\prime$. 
We then obtain an integer $N_2$ with $N_2 \ge N_2^\prime$ for which the last property of the lemma 
is valid. 


Define $\alpha_{2,1}$ to be a partition whose atom with a subscript 
$i \in \mathfrak{A}_{\alpha_{1,1}}$ is the union of atoms of $\alpha_{2,2}$ with subscripts of the 
form $(i,\ast,\ast) \in \mathfrak{A}_{\alpha_{1,1} \vee \tau_2 \vee \beta_2} = 
\mathfrak{A}_{\alpha_{2,2}}$. Since the copying procedure changes names on a set of measure 
$\delta_2$ at most, we see that $d(\alpha_{2,1},\alpha_{1,1}) \le \delta_2$, which together with 
Lemma~\ref{lem_Step_1} implies that 
\[
E_1 \overset{1/2^2+1/2^3}{\in} (\alpha_{2,1})^T.
\]
Also, it immediately follows from definition of $\alpha_{2,1}$ that 
$\alpha_{2,2} \succcurlyeq \alpha_{2,1}$. 
Since $f_{\beta_2,\alpha_{2,2}}=f_{\beta_2,\alpha_{1,1} \vee \tau_2 \vee \beta_2}$, 
$\beta_2 \succcurlyeq \beta_1$, $\alpha_{2,2} \succcurlyeq \beta_2$ and 
$\alpha_{1,1} \succcurlyeq \beta_1$, it is easy to see that 
\begin{align*}
f_{\beta_1,\alpha_{2,2}} &= f_{\beta_1,\beta_2} \circ f_{\beta_2,\alpha_{2,2}} = 
f_{\beta_1,\beta_2} \circ f_{\beta_2,\alpha_{1,1} \vee \tau_2 \vee \beta_2} \\ 
&= f_{\beta_1,\alpha_{1,1} \vee \tau_2 \vee \beta_2} \\
& = f_{\beta_1,\alpha_{1,1}} \circ f_{\alpha_{1,1}, \alpha_{1,1} \vee \tau_2 \vee \beta_2}.
\end{align*}
It follows therefore that for all $(i,j,k), (i,j^\prime, k^\prime) \in \mathfrak{A}_{\alpha_{2,2}}$, 
\[
f_{\beta_1,\alpha_{2,2}}(i,j,k) = f_{\beta_1, \alpha_{1,1}}(i) = 
f_{\beta_1,\alpha_{2,2}}(i,j^\prime,k^\prime),
\]
which shows that $\alpha_{2,1} \succcurlyeq \beta_1$. 

Since $\mathfrak{t}_2^\prime$ is a refinement of $\mathfrak{t}_1$, the copying procedure, which is 
carried out on $\mathfrak{t}_2^\prime$, guarantees that the $\alpha_{1,1}$-part of the 
$\alpha_{2,2}$-name of a given fiber of $\mathfrak{t}_1$ coincides with the $\alpha_{1,1}$-name of 
some fiber on a column where the given fiber lies. It follows therefore from definition of 
$\alpha_{2,1}$ that the $\alpha_{2,1}$-name of any fiber of $\mathfrak{t}_1$ coincides with the 
$\alpha_{1,1}$-name of some fiber of $\mathfrak{t}_1$. 
\end{proof}

Hence, the partition $\alpha_{2,1}$ has the $(N_1,1,1/2^{n_1})$-uniformity in virtue of the last 
property of Lemma~\ref{lem_Step_1}. As a consequence of the last property of Lemma~\ref{scndstplmm}, 
we know that the partition $\alpha_{2,1}$ has the $(N_2,1,1/2^{n_2})$- and 
$(N_2,2,1/2^{n_2})$-uniformity, because every finite atom of $\alpha_{2,1}$ 
(resp.\ $(\alpha_{2,1})_{-1}^1$) is a union of at most 
$(\# (\alpha_{1,1} \vee \tau_2 \vee \beta_2))^2 \times \# \tau_2 \times \# \beta_2$ (resp.\ ($\# 
(\tau_2 \vee \beta_2))^3)$ finite atoms of $(\alpha_{2,2})_{-1}^1$. Similarly, the partition 
$\alpha_{2,2}$ has the $(N_2,1,1/2^{n_2})$-uniformity. 

Continuing the inductive steps, we will find triangular array: 
\[
\Set{\alpha_{i,j}|i \in \N,1 \le j \le i}
\]
of partitions of $Y$, sequences $1 \le n_1 < n_2 < \dots $ and $1 < N_1 < N_2 < \dots$ of integers, 
and refining sequence $\Set{\mathfrak{t}_i|i \in \N}$ of towers so that for each $i \in \N$, 
\begin{enumerate}[label=(\roman*), ref=(\roman*)]
\item\label{ref_of_alpha_in_each_i}
$\alpha_{i,1} \preccurlyeq \alpha_{i,2} \preccurlyeq \dots \preccurlyeq \alpha_{i,i}$;
\item\label{ref_of_beta_than_alpha}
$\beta_j \preccurlyeq \alpha_{i,j}$ for each integer $j$ with $1 \le j \le i$;
\item\label{approx_of_base}
$E_j \overset{\epsilon_{i,j}}{\in} (\alpha_{i,j})^T$ if $1 \le j \le i$, where 
$\epsilon_{i,j} =\sum_{k=j}^i 2^{-(k+1)} = 2^{-j} - 2^{-(i+1)}$;
\item\label{Cauchy_for_alphai}
$d(\alpha_{i,j},\alpha_{i+1,j})< 2^{-n_{i+1}}$ for each integer $j$ with $1 \le j \le i$;
\item
$\mathfrak{t}_i$ is $\mathcal{F}_{\beta_i}$-measurable;
\item\label{names_of_fibers_inherited}
for each integer $j$ with $1 \le j \le i$, the $\alpha_{i,j}$-name of any fiber of $\mathfrak{t}_j$ 
coincides with the $\alpha_{j,j}$-name of some fiber of $\mathfrak{t}_j$,
\item\label{/r^2i_uniform}
if a partition $\alpha$ of $Y$ satisfies that the $\alpha$-name of any fiber of $\mathfrak{t}_i$ 
coincides with the $\alpha_{i,i}$-name of some fiber of $\mathfrak{t}_i$, then the partition $\alpha$ 
has the $(N_i,i,1/(2^{n_i}{r_i}^{2i}))$-uniformity, where 
\[
r_i = 
\begin{cases}
\# \alpha_{i,i} & \textrm{ if } i \ne 1; \\
1 & \textrm{ if } i = 1.
\end{cases}
\]
\end{enumerate} 

\begin{lemma}[Consequence of the inductive steps]\label{central_lem}
There exist a sequence $\alpha_1 \preccurlyeq \alpha_2 \preccurlyeq \ldots$ of strictly uniform 
partitions of $Y$ satisfying that $\beta_j \preccurlyeq \alpha_j$ and 
$E_j \overset{1/2^j}{\in} (\alpha_j)^T$ for every $j \in \N$.
\end{lemma}

\begin{proof}
It follows from \ref{Cauchy_for_alphai} that for each $j \in \N$, the sequence 
$\Set{\alpha_{i,j}}_{i \in \N}$ is a Cauchy sequence in $d$. For each $j \in \N$, let $\alpha_j$ 
denote a partition of $Y$ for which 
\begin{equation}\label{alphaj}
\lim_{i \to \infty}d(\alpha_j, \alpha_{i,j}) = 0.
\end{equation}
It follows from \ref{ref_of_alpha_in_each_i} and \ref{ref_of_beta_than_alpha} that for each 
$j \in \N$, $\alpha_j \preccurlyeq \alpha_{j+1}$ and $\beta_j \preccurlyeq \alpha_j$, respectively. 
It follows from \eqref{alphaj} and \ref{names_of_fibers_inherited} that for each $j \in \N$, 
the $\alpha_j$-name of any fiber of $\mathfrak{t}_j$ coincides with the $\alpha_{j,j}$-name of some 
fiber of $\mathfrak{t}_j$. This together with \ref{/r^2i_uniform} implies that each $\alpha_j$ has 
the $(N_j,j,1/(2^{n_j}{r_j}^{2j}))$-uniformity. 

Now, fix $j \in \N$. Remark that $\# \alpha_i = r_i$ for every integer $i$ with $i \ge 2$. Let 
$\epsilon > 0$ and $k \in \N$. There exists $i \in \N$ with $i \ge j \vee k$ for which 
$1/2^{n_i} < \epsilon$. Since every finite atom of $(\alpha_j)_{-k+1}^{k-1}$ is the union of at most 
${r_i}^{2i-1}$ finite atoms of $(\alpha_i)_{-i+1}^{i-1}$, the partition $\alpha_j$ has the 
$(N_i,k,\epsilon)$-uniformity. Thus, the partition $\alpha_j$ is uniform since $k$ and $\epsilon$ 
are arbitrary. Moreover, the partition $\alpha_j$ is strictly uniform, because 
$\beta_j \preccurlyeq \alpha_j$ and $\beta_j$ satisfies \eqref{inf_often_hitting}. 

The second property of the lemma is a consequence of \ref{approx_of_base} and \eqref{alphaj}. 
\end{proof}

Let $\alpha_1,\alpha_2,\dots$ be strictly uniform partitions of $Y$ as in Lemma~\ref{central_lem}. Let 
$\hat{\mathbf{X}}=(\hat{X},\hat{S})$ denote an inverse limit of an inverse system 
$\Set{(\hat{X_{\alpha_i}},\hat{S_{\alpha_i}},\phi_{\alpha_i,\alpha_{i+1}})}_{i \in \N}$, i.e.
\[
\hat{X} = \Set{(x_i)_{i \in \N} \in \prod_{i \in \N}\hat{X_{\alpha_i}}| 
x_i = \phi_{\alpha_i,\alpha_{i+1}}(x_{i+1}) \textrm{ for all } i \in \N}, 
\]
which is a Cantor set under the relative topology induced by the product topology on 
$\prod_{i \in \N}\hat{X_{\alpha_i}}$, and $\hat{S}$ is a homeomorphism on $\hat{X}$ defined by 
$(x_i)_{i \in \N} \mapsto (\hat{S_{\alpha_i}}(x_i))_{i \in \N}$. 
Since it follows from Lemma~\ref{str_unf_iff_str_erg} that the subshift 
$(\hat{X_{\alpha_i}},\hat{S_{\alpha_i}})$ is almost minimal for every $i \in \N$, it is readily 
verified that the system $\hat{\mathbf{X}}$ is almost minimal; see for more details 
\cite[Lemma~3.1]{Y}. Let $\mathbf{X}=(X,S)$ denote a minimal, locally compact Cantor system whose 
one-point compactification is the almost minimal system $\hat{\mathbf{X}}$. 

Since $\hat{\lambda_{\alpha_{i+1}}} \circ 
(\phi_{\alpha_i,\alpha_{i+1}})^{-1} = \hat{\lambda_{\alpha_i}}$ for each $i \in \N$, Kolmogorov's 
extension theorem \cite{Okabe,Yam} for infinite measures allows us to have a unique, $\sigma$-finite 
measure $\hat{\lambda}$ on the Borel $\sigma$-algebra of $\hat{X}$ satisfying a condition that 
$\hat{\lambda} \circ {p_{\alpha_i}}^{-1}=\hat{\lambda_{\alpha_i}}$ for all $i \in \N$, where 
$p_{\alpha_i}$ is the projection from $\hat{X}$ to $\hat{X_{\alpha_i}}$. The condition implies that 
$\hat{\lambda}$ is $\hat{S}$-invariant. The restriction $\lambda$ of 
$\hat{\lambda}$ to the Borel $\sigma$-algebra of $X$ is a Radon measure, because a family: 
\[
\bigcup_{i \in \N} \Set{{p_{\alpha_i}}^{-1}(E)|E \textrm{ is a compact and open subset of } 
X_{\alpha_i}}
\] 
is a base for the topology of $X$ and 
$\hat{\lambda}({p_{\alpha_i}}^{-1}(E))=\hat{\lambda_{\alpha_i}}(E) < \infty$ for all compact and open 
subset $E$ of $X_{\alpha_i}$ and $i \in \N$. Since it follows from 
Lemma~\ref{str_unf_iff_str_erg} again that for each $i \in \N$, the almost minimal system 
$(\hat{X_{\alpha_i}}, \hat{S_{\alpha_i}})$ is bi-ergodic, it is readily verified that so is 
$(\hat{X},\hat{S})$; see for more details \cite[Lemma~3.2]{Y}. Let $\hat{\mathcal{A}}$ denote the 
completion of the Borel $\sigma$-algebra of $\hat{X}$ with respect to $\hat{\lambda}$. 

Define a measurable map $\theta : Y \to \hat{X}, y \mapsto (\phi_{\alpha_i}(y))_{i \in \N}$. Since 
$\phi_{\alpha_i} = p_{\alpha_i} \circ \theta$ for all $i \in \N$, we obtain that 
$(\mu \circ \theta^{-1}) \circ {p_{\alpha_i}}^{-1} = \mu \circ {\phi_{\alpha_i}}^{-1} = 
\hat{\lambda_{\alpha_i}}$ for all $i \in \N$, so that $\mu \circ \theta^{-1}=\hat{\lambda}$ in 
virtue of the uniqueness of Kolmogorov's extension theorem, i.e.\ the map $\theta$ is 
measure-preserving. The map $\theta$ is injective, because Lemma~\ref{central_lem} shows that 
$E_i \overset{1/2^i}{\in} (\alpha_i)^T$ for every $i \in \N$ and because each member of the base 
$\mathcal{E}$ appears infinitely often in $\mathcal{E}$. Since in view of \eqref{itinerary_words}, 
the image of each $\phi_{\alpha_i}$ is written as the intersection of a decreasing sequence of those 
finite unions of cylinder subsets which have full measure with respect to $\hat{\lambda_{\alpha_i}}$, 
the image is measurable and has full measure with respect to $\hat{\lambda_{\alpha_i}}$. This implies 
that the image of the map $\theta$ has full measure with respect to $\hat{\lambda}$, so that the 
map $\theta$ is surjective. At last, we know that $\theta$ is an isomorphism between $(Y,\B,\mu,T)$ and 
$(\hat{X},\hat{\mathcal{A}},\hat{\lambda},\hat{S})$, because for any point $y \in Y$, 
\[
\theta \circ T(y)=(\phi_{\alpha_i}(Ty))_{i \in \N} =(\hat{S_{\alpha_i}}\phi_{\alpha_i}(y))_{i \in \N}
=\hat{S} \circ \theta(y).
\]

In a similar way as above, we can see that the inverse limit $\hat{\mathbf{W}}=(\hat{W},\hat{V})$ of 
an inverse system $\Set{(\hat{X_{\beta_i}},\hat{S_{\beta_i}},\phi_{\beta_i,\beta_{i+1}})}_{i \in \N}$ 
is a bi-ergodic, almost minimal Cantor system. Since 
$\hat{X_{\beta_i}}=\hat{X_{\gamma_i}}$ and $\phi_{\beta_i,\beta_{i+1}}=\phi_{\gamma_i,\gamma_{i+1}}$ 
for all $i \in \N$, the system $\hat{\mathbf{W}}$ is identical with the inverse limit of an inverse 
system $\Set{(\hat{X_{\gamma_i}},\hat{S_{\gamma_i}},\phi_{\gamma_i,\gamma_{i+1}})}_{i \in \N}$. Let 
$\hat{\xi}$ denote a unique, $\sigma$-finite measure $\hat{\xi}$ satisfying that 
$\hat{\nu} \circ {p_{\beta_i}}^{-1} = \hat{\lambda_{\beta_i}}$ for all $i \in \N$. The measure 
$\hat{\xi}$ is $\hat{V}$-invariant. Remark that for all $i \in \N$, 
$\hat{\lambda_{\beta_i}} = \hat{\lambda_{\gamma_i}}$ in virtue of a fact that $\pi$ is 
measure-preserving and $p_{\beta_i} = p_{\gamma_i}$ by definition. 
Define a map $\iota : \hat{Z} \to \hat{W}, z \mapsto (\phi_{\gamma_i}(z))_{i \in \N}$, which is 
continuous because $\iota^{-1}({p_{\gamma_i}}^{-1}([u.v]_{\hat{X_{\gamma_i}}}))=
{\phi_{\gamma_i}}^{-1}([u.v]_{\hat{X_{\gamma_i}}})$ is open for all words $u$ and $v$ satisfying that 
$uv \in \L(\gamma_i)$. The map $\iota$ is injective in virtue of the second property of 
Lemma~\ref{seq_gen_top} and surjective because the inverse image $\iota^{-1}\Set{(x_i)_{i \in \N}}$ 
of a given point $(x_i)_{i \in \N} \in \hat{W}$ is written as the intersection of decreasing sequence 
$\Set{{\phi_{\gamma_i}}^{-1}(x_i)}_{i \in \N}$ of nonempty, compact subsets of $\hat{Z}$. It follows 
that $\iota$ is a homeomorphism. It is now readily verified that $\iota$ is an isomorphism between 
bi-ergodic, almost minimal systems $\hat{\mathbf{Z}}$ and $\hat{\mathbf{W}}$. We obtain that 
$\hat{\nu} \circ \iota^{-1} = \hat{\xi}$ because 
$(\hat{\nu} \circ \iota^{-1}) \circ {p_{\gamma_i}}^{-1} = \hat{\nu} \circ {\phi_{\gamma_i}}^{-1} = 
\hat{\lambda_{\gamma_i}}$ for all $i \in \N$. 

Define a continuous map $\kappa : \hat{X} \to \hat{W}, (x_i)_{i \in \N} \mapsto 
(\phi_{\beta_i,\alpha_i}(x_i))_{i \in \N}$, which is verified to be surjective in virtue of a fact 
that for all $i \in \N$, 
\[
\phi_{\beta_i,\beta_{i+1}} \circ \phi_{\beta_{i+1},\alpha_{i+1}}= \phi_{\beta_i,\alpha_i} \circ 
\phi_{\alpha_i,\alpha_{i+1}}.
\]
By using a fact that for all $i \in \N$,
\[
\phi_{\beta_i,\alpha_i} \circ \hat{S_{\alpha_i}}= \hat{S_{\beta_i}} \circ \phi_{\beta_i,\alpha_i},
\]
we see that $\kappa$ is a factor map from $\hat{\mathbf{X}}$ to $\hat{\mathbf{W}}$. The map $\kappa$ 
is measure-preserving because for all $i \in \N$, 
\[
(\hat{\lambda} \circ \kappa^{-1}) \circ {p_{\beta_i}}^{-1} = 
\hat{\lambda} \circ {p_{\alpha_i}}^{-1} \circ (\phi_{\beta_i,\alpha_i})^{-1} = 
\hat{\lambda_{\alpha_i}} \circ (\phi_{\beta_i,\alpha_i})^{-1} = \hat{\lambda_{\beta_i}}.
\]
We obtain that for any point $y \in Y$,
\[
\kappa \circ \theta(y)=(\phi_{\beta_i,\alpha_i}(\phi_{\alpha_i}(y)))_{i \in \N} =
(\phi_{\beta_i}(y))_{i \in \N} = (\phi_{\gamma_i}(\pi(y)))_{i \in \N} = \iota(\pi(y)).
\]
Putting $\rho = (\iota^{-1} \circ \kappa)|_X$ completes the proof of Theorem~\ref{ICMT}. 

\section{Proof of Theorem~\ref{categorical_consequence}}\label{pf_for_thm1.2}

In view of Theorem~\ref{ICMT}, it is sufficient to prove the last assertion of 
Theorem~\ref{categorical_consequence}. We shall now start the proof with the following lemma. 
\begin{lemma}\label{disjoint}
Suppose that $\mathbf{X}_i=(X_i,S_i)$ is a locally compact Cantor system for each $i \in \Set{1,2,3}$. 
Suppose that $p_i$ is a proper factor map from $\mathbf{X}_1$ to $\mathbf{X}_i$ for each $i \in 
\Set{2,3}$. If $\mathbf{X}_1$ is strictly ergodic and 
${p_2}^{-1} (U_2) \cap {p_3}^{-1} (U_3) \ne \emptyset$ for all nonempty open sets $U_2 \subset X_2$ 
and $U_3 \subset X_3$, then a locally compact Cantor system 
$\mathbf{X}_2 \times \mathbf{X}_3:=(X_2 \times X_3,S_2 \times S_3)$ is strictly ergodic. 
\end{lemma}

\begin{proof}
For each $i \in \Set{1,2,3}$, let $\hat{\mathbf{X}_i}=(\hat{X_i},\hat{S_i})$ be the one-point 
compactification of $\mathbf{X}_i$ and $\hat{X_i}=X_i \cup \Set{\omega_i}$. We have another 
one-point compactification $\widehat{X_2 \times X_3}=(X_2 \times X_3) \cup \Set{(\omega_2,\omega_3)}$ 
of a locally compact Cantor set $X_2 \times X_3$, which gives a topological dynamical system 
$\widehat{\mathbf{X}_2 \times \mathbf{X}_3}=(\widehat{X_2 \times X_3},\widehat{S_2 \times S_3})$. 
Define a map $\widehat{p_2 \times p_3}:\hat{X_1} \to \widehat{X_2 \times X_3}$ by 
\[
\widehat{p_2 \times p_3}(x)=
\begin{cases}
(p_2(x),p_3(x)) & \textrm{ if } x \in X_1; \\
(\omega_2,\omega_3) & \textrm{ if } x = \omega_1.
\end{cases}
\]
Since the maps $X_1 \to X_1 \times X_1,x \mapsto (x,x)$ and $X_1 \times X_1 \to X_2 \times X_3, 
(x,x^\prime) \mapsto (p_2(x),p_3(x^\prime))$ are both proper, so is their composition 
$p_2 \times p_3:X_1 \to X_2 \times X_3,x \mapsto (p_2(x),p_3(x))$. Hence, the map 
$\widehat{p_2 \times p_3}$ is continuous. 
If $\widehat{p_2 \times p_3}$ is not surjective, then there exist a nonempty open subset 
$U$ of $\widehat{X_2 \times X_3}$, which is disjoint from the point $(\omega_2,\omega_3)$, such that 
$\widehat{p_2 \times p_3}^{-1}(U)=\emptyset$. However, it is clearly impossible in view of the 
hypothesis of the lemma. Hence, the map $\widehat{p_2 \times p_3}$ is a factor map from 
$\hat{\mathbf{X}_1}$ to $\widehat{\mathbf{X}_2 \times \mathbf{X}_3}$. 
Corollary~\ref{factor_of_unq_LCCM} completes the proof. 
\end{proof}

Assume that $\mathbf{Y}=(Y,\B,\mu,T)$ is an ergodic, infinite measure-preserving system whose cartesian 
product $\mathbf{Y} \times \mathbf{Y}$ with itself is ergodic; see \cite{MR0153815} for the existence 
of such systems. In Figure~\ref{inverted_tree}, let $\mathbf{Y}_2=\mathbf{Y}_3=\mathbf{Y}$, 
$\mathbf{Y}_1 = \mathbf{Y}_2 \times \mathbf{Y}_3$ and $q_i: \mathbf{Y}_1 \to \mathbf{Y}_i$ the 
projection for each $i \in \Set{2,3}$. Let $\gr(\mu,\id)$ denote the diagonal measure on $\mathbf{Y} 
\times \mathbf{Y}$, i.e.\ $\gr(\mu,\id)(A \times B)=\mu(A \cap B)$ for all sets $A,B \in \B$. 

Assume now that the diagram in Figure~\ref{inverted_tree} has a strictly ergodic, locally compact 
Cantor model:
\begin{center}
\includegraphics[scale=0.85]{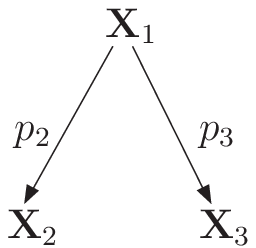}
\end{center}
For each $i \in \Set{1,2,3}$, let $\sigma_i$ denote the relevant isomorphism from $\mathbf{Y}_i$ to 
$\mathbf{X}_i=:(X_i,S_i)$. Suppose that $U_i \subset X_i$ is nonempty and open for each 
$i \in \Set{2,3}$. Since $\mu({\sigma_i}^{-1}(U_i))>0$ for each $i \in \Set{2,3}$, we know that 
\[
(\mu \times \mu)({q_2}^{-1}{\sigma_2}^{-1} (U_2) \cap {q_3}^{-1}{\sigma_3}^{-1} (U_3)) > 0,
\]
and hence, ${p_2}^{-1}(U_2) \cap {p_3}^{-1}(U_3) \ne \emptyset$. It follows from Lemma~\ref{disjoint} 
that $\mathbf{X}_2 \times \mathbf{X}_3$ is strictly ergodic. This leads to a contradiction that 
$\mu \times \mu$ coincides, up to a positive constant multiple, with the diagonal measure 
$\gr(\mu,\id)$. This completes the proof of Theorem~\ref{categorical_consequence}. 

\bigskip

\paragraph{Acknowledgement}
This work was partially supported by JSPS KAKENHI Grant Number 15K04900.

\end{document}